\documentclass{amsart}
\usepackage[utf8]{inputenc}

\usepackage{geometry}
\usepackage{amsmath,amsthm,amssymb}
\usepackage{bbm}
\usepackage{graphicx}
\usepackage{subcaption}
\usepackage{todonotes}
\usepackage{cleveref}
\usepackage{color}
\usepackage[style=ext-numeric,sorting=nyt, url = false, doi = false, eprint = false, isbn = false,giveninits=true, articlein=false,uniquename=init]{biblatex}

\usepackage{cancel}
\numberwithin{equation}{section}

\newtheorem{theorem}{Theorem}[section]

\newtheorem{defn}[theorem]{Definition}
\newtheorem{prop}[theorem]{Proposition}
\newtheorem{lemma}[theorem]{Lemma}

\newtheorem{remark}[theorem]{Remark}

\newcommand{\D}{\cancel{D}}

\newcommand{\eps}{\varepsilon}

\newcommand{\re}{\mathrm{Re} \, }
\newcommand{\tr}{\mathrm{Tr} \, }

\newcommand{\EE}{\tilde{B}}
\newcommand{\FF}{\tilde{A}}
\newcommand{\GG}{\tilde{D}}
\newcommand{\HH}{\tilde{C}}
\newcommand{\R}{\mathbb{R}}

\newcommand{\Z}{\mathbb{Z}}
\newcommand{\C}{\mathbb{C}}

\newcommand{\alo}{\alpha_l^{(1)}}
\newcommand{\alt}{\alpha_l^{(2)}}
\newcommand{\ali}{\alpha_l^{(i)}}
\newcommand{\alii}{\alpha_l^{(ii)}}

\newcommand{\Reso}{\mathfrak{R}}
\newcommand{\Resk}{\mathcal{R}}
\newcommand{\lo}{\lambda(\omega)}
\newcommand{\jostw}{\eta}

\newcommand{\jostk}{\xi}

\newcommand\numberthis{\addtocounter{equation}{1}\tag{\theequation}}

\title[Dirac decay estimates]{Dispersive decay estimates for Dirac equations with a domain wall}

\author{Joseph Kraisler}
\address{Department of Mathematics and Statistics, Amherst College, 31 Quadrangle Dr, Amherst, MA 01002, USA}
\email{jkraisler@amherst.edu}
\author{Amir Sagiv}
\address{Faculty of Mathematics, Technion - Israel Institute of Technology, Haifa 32000, Israel}
\email{amirsagiv@technion.ac.il}
\author{Michael I.\ Weinstein}
\address{Department of Applied Physics and Applied Mathematics and Department of Mathematics, Columbia University, 500 W120 Street, New York, NY 10027, USA}
\email{miw2103@columbia.edu}
\begin{document}
\maketitle
\begin{abstract}
Dispersive time-decay estimates are proved for a one-parameter
family of one-dimensional Dirac Hamiltonians with dislocations; these are operators which  interpolate between two phase-shifted massive Dirac Hamiltonians at $x=+\infty$ and $x=-\infty$.  This family of Hamiltonians arises in the theory of topologically protected states of one-dimensional quantum materials.
For certain values of the phase-shift parameter, $\tau$, the Dirac Hamiltonian has a {\it threshold resonance} at the endpoint of its essential spectrum. Such resonances are known to influence the time-decay rate.  Our main result  explicitly displays the transition in time-decay rate as $\tau$ varies between resonant and non-resonant values.
Our results appear to be the first dispersive time-decay estimates for Dirac Hamiltonians which are not a relatively compact perturbation of a free Dirac operator.
\medskip

\end{abstract}

%%%%%%%%%%%%%%%%%%%%%%%%%%%%%%%%%%%%%%%%%%%%%%%%%%%%%%%%%%%%%%%%%%%%%%%%%%
\section{Introduction} \label{sec:intro}
%%%%%%%%%%%%%%%%%%%%%%%%%%%%%%%%%%%%%%%%%%%%%%%%%%%%%%%%%%%%%%%%%%%%%%%%%%
We study the time-dynamics for a family of Dirac Hamiltonians, $\{\D(\tau)\}_{\tau\in S^1}$,  which arise in the theory of topologically protected states in one-dimensional asymptotically periodic   quantum systems with a dislocation defect (domain wall). The parameter $\tau\in S^1=\R/2\pi\Z$ specifies the size of the dislocation. 
  We now introduce the model 
  and in Section \ref{sec:top}  below, we discuss a context in which this class of models plays a central role.

Consider the following initial-value problem, describing the evolution of complex vector amplitude, $\alpha(t,x):\R_t\times \R_x \to \C^2$:
\begin{subequations}\label{eq:Diraceq}
\begin{equation}
    i\partial_t \alpha = \D (\tau)\alpha \, , 
 \qquad \alpha(0,x)=\alpha_0 (x)\in L^2 (\R ;\C^2) \, .
 \end{equation}
For each $\tau\in \R/2\pi\Z$, the  Dirac Hamiltonian, $\D(\tau)$, is given by:
\begin{align}\label{eq:Dtau}
 \D (\tau) &\equiv i\sigma_3\partial_x  + \sigma_1 \mathbbm{1}_{(-\infty,0)}(x)  + \sigma_\star(\tau)\mathbbm{1}_{[0,\infty)}(x)\ .
\end{align}
\end{subequations}
Here, $\sigma_j$ denote standard Pauli matrices (see \eqref{eq:Pauli}),
\begin{equation*}%\label{eq:sigma_st}
\sigma_\star(\tau) \equiv \begin{pmatrix}
    0 & e^{-i\tau} \\ e^{i\tau} & 0
    \end{pmatrix}\ ,
\end{equation*} and 
  $\mathbbm{1}_{S}(x)$ denote the indicator function of a set $S\subseteq \R$. The curve $\tau\in S^1\mapsto\D(\tau)$ sweeps out a periodic  family of self-adjoint operators on $L^2(\R;\C^2)$ with common domain $H^1(\R;\C^2)$.  

{\it The goal of this paper is to provide a detailed understanding of the time-dynamics \eqref{eq:Diraceq} as the dislocation parameter $\tau$ varies; in particular the dispersive decay of solutions, as $t\to \infty$,  for initial data projected onto the continuous spectral part of $ \D (\tau)$.}

 Note that  
 \begin{equation*}%\label{eq:Dtau2}
 \D (\tau) =\begin{cases}
 \D_+(\tau) \equiv i\sigma_3\partial_x  + \sigma_\star(\tau) & \textrm{for $x>0$}\\
  \D_-(\tau) \equiv i\sigma_3\partial_x  + \sigma_1   &  \textrm{for $x<0$}.
 \end{cases}
 \end{equation*}
If $\tau=0$ (or $\tau=2\pi$), $\D_+(0)=\D_-(0)$, and  $\D (0)$ is a constant coefficient (i.e., translation invariant) massive Dirac Hamiltonian. In contrast, for all 
 $\tau\notin\{0,2\pi\}$, the operator $\D_+(\tau)$ is non-trivially ``phase-shifted" from
  $\D_-(\tau)$. Such a defect is called a {\it domain wall} or a {\it dislocation}. 
  
Since $\D(0)=\D (2\pi)$ commutes with spatial translations, its spectrum consists entirely of essential 
 (continuous) spectrum, with a gap consisting of all energies in $(-1,+1)$. For $\tau \neq 0,2\pi$, however, $\D(\tau)$ has a defect-mode: an eigenvalue in this spectral gap, spatially localized around $x=0$.   
 The spectral properties of $\D(\tau)$ are well-known (see \cite{drouot2021bulk, drouot2020defect, fefferman2017topologically}) and are summarized in the following (also a consequence of  Proposition \ref{prop:eigs-poles} below):
 \begin{theorem}
     \label{prop:specD}
\begin{enumerate}
    \item For all $\tau\in[0,2\pi]$, the essential spectrum of $\D(\tau)$ is given by:
   \[ \sigma_{\rm ess}(\D(\tau)) = (-\infty,-1]\cup [1,\infty).\]
    \item For $\tau\in[0,2\pi]$, 
     \[ \sigma(\D(\tau)) = \begin{cases}
         (-\infty,-1]\ \cup\ \{\omega_{\tau}\}\ \cup\ [1,\infty),\quad \tau\in(0,2\pi)\\
         (-\infty,-1]\ \cup\ [1,\infty),\quad \tau=0, 2\pi
     \end{cases}.\]
     Here, $\omega_{\tau}$ is an eigenvalue of multiplicity one within the spectral gap (the gap within the essential spectrum) and is given by the expression:
     \[ \omega_{\tau} = \cos (\tau /2 ) \, . \]
      with corresponding one-dimensional eigenspace spanned by
          \begin{equation}
      \psi_\tau(x) = \frac12 \sqrt{\sin\left(\frac{\tau}{2}\right)} e^{-\sin\left(\frac{\tau}{2}\right)|x|}
      \begin{pmatrix}
          1 \\ -e^{i\frac{\tau}{2}}
      \end{pmatrix},\quad \|\psi_\tau\|_{L^2}=1.
    \label{eq:alpha*}  \end{equation}

\end{enumerate}
 \end{theorem}
That $\omega_{\tau}$ traverses the spectral gap as $\tau$ varies from $0$ to $2\pi$  has an underlying topological explanation, which we discuss in 
Section \ref{sec:top}.

 For $\tau =0, 2\pi$, the Hamiltonian $\D(\tau)$ is a translation invariant (constant coefficient) differential operator. Each component of $\alpha$ satisfies a Klein-Gordon equation: $\partial_t^2 U(x,t)= (\partial_x^2-1) U(x,t)$. Hence, for $\tau =0,2\pi$, all sufficiently smooth and spatially localized initial conditions disperse to zero (spread and decay)  as time advances under the evolution \eqref{eq:Diraceq} \cite{brenner1985scattering}.
 
 In contrast, by Theorem \ref{prop:specD}, if $\tau \in (0,2\pi)$, one expects dispersive time-decay only after projecting onto the continuous spectral subspace  of  $\D(\tau)$.  We denote this projection by $P_{ac}(\D(\tau))$. By Theorem \ref{prop:specD} 
 \[ P_{ac}(\D(\tau))= I - \langle\psi_{\tau},\cdot\rangle\ \psi_{\tau} .\]

 \noindent
Our main result is the following time-decay estimate, with a $\tau-$ dependent rate of decay.

\begin{theorem}\label{maintheorem}
Let $\tau\in [0,2\pi]$ and $\D(\tau)$ be as defined in (\ref{eq:Dtau}). For any $\epsilon >0$, there exists $C_\epsilon>0$, which is independent of $\tau$, such that for all  we have
\begin{equation}\label{eq:decayest1}
    \|\langle x\rangle^{-2}e^{-i\D(\tau)t}P_{ac}(\D(\tau))\langle \D(\tau)\rangle^{-3/2-\epsilon}\langle x\rangle^{-2}\|_{L^1\to L^{\infty}}\leq C_\epsilon\frac{1}{\langle t\rangle^{1/2}}\frac{1}{1+\sin^2(\tau/2) t} .
\end{equation}
\end{theorem}

From the bound \eqref{eq:decayest1} we note the non-uniformity of the time scale on which $t^{-\frac32}$ occurs  for $\tau$ or $2\pi-\tau\ll1\ ({\rm mod}\ 2\pi)$; indeed the $t^{-3/2}$ rate of decay only becomes visible on a time scale $t\gg \tau^{-2}$. Furthermore, while there is discontinuity in the $t\to \infty$ decay rate at $\tau=0$, for a fixed time \eqref{eq:decayest1} shows that the decay rate is continuous with respect to $\tau$ even at $\tau=0$.

A dispersive decay bound which displays a similar transition was obtained for the time-dependent Schr\"odinger equation with highly oscillatory potential in \cite{OscillatoryPotentials}.

\begin{remark}[Smoothness, spatial localization
 and uniformity of \eqref{eq:decayest1} in $\tau$]
  The decay bound \eqref{eq:decayest1} can be written equivalently as 
  \begin{equation}\label{eq:decayest2}
    \|\langle x\rangle^{-2}e^{-i\D(\tau)t}P_{ac}(\D(\tau))\alpha_0\|_{L^\infty(\R)}\leq C_\epsilon\frac{1}{\langle t\rangle^{1/2}}\frac{1}{1+\sin^2(\tau/2) t}
    \|\langle x\rangle^{2} \langle \D(\tau)\rangle^{3/2+\epsilon}\alpha_0\|_{L^1(\R)} \, .
\end{equation}
 This time-decay bound requires both smoothness and spatial localization of the initial condition. The smoothness required on the right hand side of \cref{eq:decayest2} arises from high-energy considerations; for the high-energy part of the data, the dynamics are comparable to those of the Klein-Gordon equation, whose dispersion relation $\omega(k)=\sqrt{1+k^2}$ has a dispersion rate, $\omega''(k)$, which tends to zero as~$k \to \infty$.\footnote{For the Klein-Gordon equation $k$ is simply the Fourier variable.} Smoothing provides an effective truncation of the spectral support of the data~$\alpha_0$, leading to an effective lower bound on the dispersion rate.

  Further, the  localization in $x$  arises from seeking a upper bound for the time-decay rate which holds {\em uniformly} in $\tau\in[0,2\pi]$, i.e., the constant $C_\epsilon$ in \eqref{eq:decayest1} should be independent of $\tau$. This boils down to an analysis of the behavior of the resolvent of $\D(\tau)$ for energies near the endpoints (thresholds) of the essential spectrum, for which we require spatial weights which are at least quadratic in~$x$. It is an open problem whether the spatial weights can be reduced to $\langle x \rangle^{-1}$.
\end{remark}

\begin{remark}[Threshold resonance phenomena]\label{rem:thresh}

 While our initial consideration of $\D (\tau)$ was motivated by its role in applications to periodic structures with dislocations (see Section~\ref{sec:motivation} below), Theorem \ref{maintheorem} also turns out to be a useful and simple paradigm for illustrating dispersive wave phenomena: The transition in decay rate in the estimate \eqref{eq:decayest1} reflects the appearance of a {\it threshold resonance} just as the eigenvalue $\omega_{\tau}$, when $\tau\equiv 0\ ({\rm mod} 2\pi)$, reaches the endpoint of the essential spectrum. The threshold resonance phenomenon, associated with the emergence of point spectrum from (or disappearance into) the continuous spectrum, is well-known to impact the rate of local energy decay in wave equations. This is best known in the context of the Schr\"odinger operator with a spatially localized potential; see, for example, \cite{jensen1979spectral, komech2010weighted, kopylova2014dispersion,SchlagSurvey}. 
The effect of threshold resonances for one-dimensional massive Dirac equations with a spatially localized potential is discussed in \cite{ERDOGAN1D}. Our family of Dirac operators is,  to the best our knowledge, the only example of Dirac operators where the transition in decay rate, due to a threshold resonance, is explicitly displayed. Such a transition in decay rates  has displayed as well for a class of Schr\"odinger Hamiltonians with rapidly oscillatory potentials in \cite{OscillatoryPotentials}. 
\end{remark}

\begin{remark}[Relation to previous time-decay results for  Dirac Hamiltonians]

Time-decay estimates for  massive Dirac operators for classes of potentials, which decay to zero as $x\to\pm\infty$, were obtained by Erdogan and Green one space dimension in  \cite{ERDOGAN1D}. For results in two and three space dimensions, see e.g., \cite{d2005decay, burak2019limiting, erdougan2021massless, erdougan2019dispersive, erdougan2018dispersive, ERDOGAN2DMassive, kovavrik2022spectral}. Pelinovsky and Stefanov derived estimates for the linearized evolution about standing waves of a one-dimensional semilinear Dirac equation in \cite{pelinovsky2012asymptotic}.

In this remark, we make two contrasts of the results and methods of this article with 
 those in~\cite{ERDOGAN1D}. First, the operators we study, $\D(\tau)$, are not spatially localized perturbations of a constant coefficient operator.
Second, our strategy of proof is different in several ways (e.g., in developing the scattering theory of $\D (\tau)$), see Section \ref{sec:strategy}.
\end{remark}

The following theorem states alternative variations of time-decay upper bounds, which hold
for various relaxations of the assumptions on spatial localization of the initial conditions.
\begin{theorem}\label{notmaintheorem}
\begin{enumerate}
    \item For any $\tau \in [0,2\pi]$ and $\varepsilon>0$, there exists a constant, $C_{\varepsilon} (\tau)>0$, such that
    \begin{equation}\label{eq:noweights} \|e^{-i\D(\tau)t}P_{ac}(\D(\tau))\langle \D(\tau)\rangle^{-3/2-\eps}\|_{L^1\to L^{\infty}} \leq C_{\varepsilon}(\tau)\frac{1}{\langle t\rangle^{1/2}} \, .
    \end{equation}
    \item Let $K\subset(0,2\pi)$ be a compact set and fix $\varepsilon>0$. Then there exists a constant $C_1(K)>0$ such that for any $\tau\in K$ and $t>0$ we have
\begin{equation}\label{eq:betterWeight}
    \|\langle x\rangle^{-1}e^{-i\D(\tau)t}P_{ac}(\D(\tau))\langle \D(\tau)\rangle^{-3/2-\eps}\langle x\rangle^{-1}\|_{L^1\to L^{\infty}} \leq C_1(K)\frac{1}{\langle t\rangle^{3/2}} \, , 
\end{equation}
and $C_1 (K_n)\to \infty$ for every sequence of compact sets such that $K_n \to (0,2\pi)$.
\end{enumerate}
\end{theorem}

\begin{remark}[Generalizations of Theorems \ref{maintheorem} and \ref{notmaintheorem}]
 \label{sec:generalize}
It is natural to consider extending Theorems \ref{maintheorem} and \ref{notmaintheorem}  to the case where $\D(\tau)$ in \eqref{eq:Dtau} is replaced by an operator which continuously interpolates between the asymptotic operators  $\D_-$ and $\D_+$. In this general setting we expect $\D(\tau)$ to have, for each fixed $\tau$, a finite number of eigenvalues in the spectral gap and also possible threshold resonance energies
at the edge of the essential spectrum. Analogous decay estimates to those presented in Theorems \ref{maintheorem} and \ref{notmaintheorem} would hold with (a)  $P_{ac}(\D(\tau))$ defined as the projection orthogonal to all bound states corresponding to the energies in the spectral gap and (b) decay rates adjusted to transition from 
 the faster $\langle t\rangle^{-\frac32}$ rate, for typical values of $\tau$, to the slower $\langle t\rangle^{-\frac12}$ rate arising from those values of $\tau$ for which a threshold resonance occurs. See the further discussion in Section \ref{sec:open} below. 
\end{remark}
\begin{remark}In the resonant case $\tau=0,2\pi$, which is in fact the standard massive Dirac equation, a $t^{-3/2}$ decay rate holds true for a variation of \eqref{eq:betterWeight}. Essentially, if one subtracts a finite-rank operator from the evolution equation, one can counteract the effect of resonance and obtain a non-resonant type order of decay. For details, see \cite[Theorem 1.2]{ERDOGAN1D}. 
\end{remark}

\subsection{Motivation;  topologically protected states and radiation damping for Floquet systems}\label{sec:motivation}
\subsubsection{Topologically protected states in dislocated systems}\label{sec:top}
Since the experimental observation of the integer quantum Hall effect \cite{cage2012quantum, klitzing1980new}, and its subsequent explanation using the topological phases of matter \cite{thouless1982quantized},
there has been a great deal of fundamental and applied interest in  wave systems having properties which are ``topologically protected'', i.e., systems in which energy transport remains unchanged under continuous localized deformations, due to a topological invariant; see, for example, the reviews \cite{hasan2010colloquium, ozawa2019topological, xue2022topological} and references therein.

The Dirac operators which this article studies, $\{\D(\tau)\}_{\tau \in \R/2\pi\Z}$, play a central role in a class of one-dimensional PDE models of physical media which exhibit topological properties. We briefly describe these models, and refer to \cite{drouot2021bulk, drouot2020defect, fefferman2017topologically, fefferman2014topologically} for details. For a photonic realization of such systems, see \cite{photonic-PRA-2016}. Consider a Schr{\"o}dinger operator, $H^{(0)}$, with a one-periodic potential, $V(x)$. Discrete translation-invariance of $H^{(0)}$ implies a {\it band structure}; its spectrum is characterized by a sequence of {\it eigenvalue dispersion curves}: $k\in[0,2\pi]\mapsto E(k)$ (Floquet-Bloch band functions).
% and corresponding $k-$pseudoperiodic eigenfunctions.
 A scenario of great is interest is when 
two such curves meet at a pair $(k_D,E_D)$ and form a linear crossing, a {\it Dirac point} \cite{fefferman2017topologically}. In this case, $E_D$ is a two-fold degenerate eigenvalue of $H^{(0)}$. The effective (or homogenized) Hamiltonian which describes a blowup of the band structure near such linear crossings is $H^{(0)}_{\rm eff}=i\sigma_3\partial_X$, a massless Dirac Hamiltonian. 
 
Next, we introduce a {\it dislocation},
a perturbing potential which is also asymptotically one-periodic but, as $x\to+\infty$ is phase-shifted by an amount $\tau/2\pi$ from its behavior as $x\to~-\infty$. The parameter   $\tau\in[0,2\pi]$ measures the amount of dislocation. The corresponding Schr{\"o}dinger operator, $H^{(\tau)}$, has a gap in its essential spectrum about energy $E_D$ (where two bands of spectra of $H^0$ cross linearly). Typically, $H^{(\tau)}$ will have a finite number of eigenvalues, $E^{(\tau)}_j$, which ``flow'' within the spectral gap as $\tau$ varies. Figure \ref{fig:SF} displays several examples where eigenvalues traverse the gap 
 as $\tau$ varies from $0$ to $2\pi$.

The  {\it spectral flow} of the family of operators $\{H^{(\tau)}\}_{\tau\in S^1}$ is an integer which  measures the net number of edge state curves which traverse the spectral gap. The spectral flow is a topological invariant; it does not change under continuous deformations of the Hamiltonian under which the gap remains open \cite{waterstraat2016fredholm}. In \cite{drouot2021bulk}, $H^{(\tau)}$ is embedded in a continuously varying family of Hamiltonians  
 $ s\in(0,1]\mapsto H^{(\tau,s)}$, where $H^{(\tau,1)}=H^{(\tau)}$.
In the asymptotic regime, $s \ll 1$, the operators $H^{(\tau , s)}$ encode a small and adiabatic dislocation. By a multiple-scales analysis \cite{drouot2020defect, fefferman2014topologically,fefferman2017topologically}, the study of the spectral flow of  $\{H^{(\tau , s)}\}_{\tau\in S^1}$ can be reduced to that of Dirac operators of the form:
 \begin{align}\label{eq:Dtau-m}
 \D(m\tau) &\equiv i\sigma_3\partial_X  + \sigma_1 \mathbbm{1}_{(-\infty,0)}(X)  + \begin{pmatrix}
    0 & e^{-im\tau}\\ e^{im\tau} & 0
    \end{pmatrix}\ \mathbbm{1}_{[0,\infty)}(X)\ .
\end{align}
The integer $m$ denotes the winding number about zero of a complex-valued  ``effective mass'' parameter which emerges from the multiple scale analysis.  By Theorem \ref{prop:specD}, the spectral flow of the family $\{\D(m\tau)\}_{\tau\in \R/(2\pi/m)}$ is equal to $1$ and hence the spectral flow of the family $\{\D(m\tau)\}_{\tau\in \R/2\pi}$ is equal to $m$. We conclude that, provided that a gap remains open throughout the deformation, the net number of 
 eigenvalue curves of $H^{(\tau)}$ which traverse the spectral gap is equal to $m$ (see Figure \ref{fig:SF}).  Each operator family $\{\D(m\tau)\}_{\tau\in S^1}$ is a canonical operator (or normal form) to which all operator families $\{H^{(\tau)}\}_{\tau\in S^1}$ in the same topological class can be deformed; see right panel of Figure \ref{fig:SF}.

Finally,  Drouot proves in \cite{drouot2021bulk} that the spectral flow, $m$, can also be identified as the first Chern number, $\mathcal{C}_1(\mathcal E)$, associated with a vector bundle over $\mathbb{T}^2_{k,\tau}$, whose fibers are the $k-$pseudoperiodic eigenspaces 
of the periodic Schr\"odinger operator at infinity $H^{(\tau)}_+$ associated with all bands below the energy $E_D$, which lies in a gap. This equality of bulk and edge indices is a variant of the {\it bulk-edge correspondence} principle. 

\begin{figure}[h]
 \begin{subfigure}[t]{0.3\textwidth}
            \includegraphics[width=\linewidth]{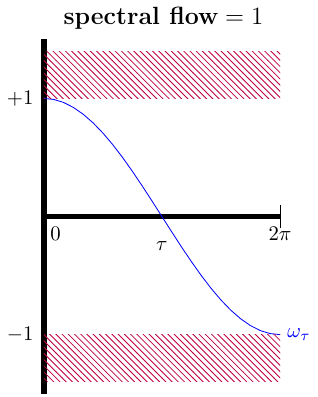}
                    \end{subfigure}
        \begin{subfigure}[t]{0.3\textwidth}
            \includegraphics[width=\linewidth]{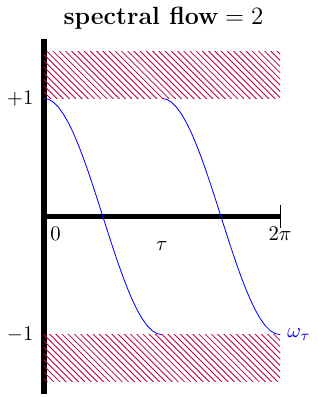}
                    \end{subfigure}
        \begin{subfigure}[t]{0.3\textwidth}
            \includegraphics[width=\linewidth]{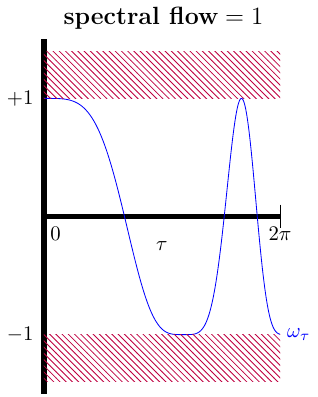}
                    \end{subfigure}
\caption{
Spectral flow illustrated (see Section \ref{sec:top}). All figures show energy (spectrum) as a function of the dislocation parameter $\tau$: continuous spectrum (red stripes) and curves of point spectra (solid blue). {\bf Left:} $\D (m\tau)$ with $m=1$, see \eqref{eq:Dtau-m}. {\bf Center:} same, with $m=2$. {\bf Right:} a continuous deformation of the $m=1$ case in which $\tau$ is replaced by a curve $\theta(\tau)$, which does not close the gap. The curve $\omega_{\tau}$, the energy of the eigenfunction of $\D (\theta(\tau))$, is non-monotonic and the spectral flow is $1$. }
\label{fig:SF}
\end{figure}

\subsubsection{Radiative decay and metastability of edge modes}
\label{sec:rad-damp}
In \cite{hameedi2022radiative},  
 motivated by modeling and experiments in photonic waveguides \cite{bellec2017non, jurgensen2021quantized}, 
  parametric forcing perturbations of the dynamics  $i\partial_t \alpha = \D (\pi)\alpha$ is studied (i.e., linear and weakly time-dependent perturbations of $\D (\pi)$). In particular, it is shown via an asymptotic expansion of solutions and detailed comparison with computer simulations, that the zero energy eigenstate of $\D (\pi)$ is only {\em meta-stable} under sufficiently rapid forcing; we demonstrate it decays exponentially in $t$ for  $0<t\lesssim  \beta^{-2}$, for $0<\beta\ll1$,   the amplitude  of the forcing.\footnote{It is conjectured that for times $t\gg\beta^{-2}$, the decay is at an algebraic rate. See analogous results concerning the Schr{\"o}dinger equation in \cite{soffer1998nonautonomous}.}

To establish a systematic multiple-scale expansion of the solution, we required a weak form of the local energy decay estimate of Theorem \ref{maintheorem}.
The  approach we present here implies such an estimate by a more direct and generalizable method than in \cite{hameedi2022radiative}. The latter involved first relating the problem to a Schr{\"o}dinger evolution, by squaring the Dirac Hamiltonian, and then  applying results on the boundedness of wave operators associated with Schr\"odinger equation \cite{weder2000inverse, yajima1995wk}.

\subsection{Open questions and future directions}\label{sec:open}
\begin{enumerate}
\item We believe that the time-decay estimates in Theorems \ref{maintheorem} and \ref{notmaintheorem} can be extended to the case where $\D(\tau)$ in \eqref{eq:Dtau} is replaced by a Dirac  operator with a domain wall, which continuously interpolates between the asymptotic operators  $\D_-$ and~$\D_+$ (as opposed to the discontinuity present in the definition of $\D (\tau)$, \eqref{eq:Dtau}); see also the discussion in Remark \ref{sec:generalize}.
One approach to such a result would be to  modify the present analysis by working with the {\it distorted plane waves}  derivable from appropriate linear integral equations; in the present work we used explicit expressions for distorted plane waves, obtainable because $\D(\tau)$ has piecewise constant coefficients. Alternatively, if $\tilde{\D}(\tau)$ is a Dirac operator with continuously varying coefficients, and such that $\tilde{\D}(\tau)\to\D(\tau) $ sufficiently rapidly as $x\to\pm\infty$, then 
 $\tilde{\D}(\tau)$ is a relatively compact perturbation of $\D(\tau)$, and one can seek to derive time-decay estimates for 
  $\exp(-it\tilde{\D}(\tau))P_{\rm ac}(\tilde{\D}(\tau))$ by a perturbation argument.

\item A fully rigorous justification of the expansion constructed in \cite{hameedi2022radiative}, which describes the radiation damping of the localized mode (see Section \ref{sec:rad-damp}), is  an interesting open problem. 
We believe that such a proof would require time-decay estimates for the {\em full  Floquet (time-periodic) Hamiltonian} with a domain wall, where the time-dependent forcing is not spatially localized. Such estimates present a challenging open analytical problem, and do not follow from known results on decay estimates for time-periodic Hamiltonians, which impose spatial localization conditions on the forcing \cite{beceanu2011new, galtbayar2004local, soffer2020p}.

\item  Section \ref{sec:top} outlines a context in which Dirac operators with domain wall / dislocation defects arise and their role in determining topological properties. 
In physical systems, such as in optics (e.g. \cite{photonic-PRA-2016}), nonlinear corrections to the underlying linear wave equations are needed to model effects at high intensities. For the systems described in Section \ref{sec:top}, the corresponding effective Dirac equations  have a nonlinearity, analogous to the Kerr (cubic) nonlinearity of the nonlinear Schr\"odinger equation. 
A natural agenda of analytical questions for the semi-linear Dirac models with domain walls ranges from low energy scattering for systems without a domain wall defect,  nonlinear scattering and asymptotic stability theory for nonlinear bound states which bifurcate from the linear spectrum of Dirac operators with domain walls, to radiation damping and metastability of states.   This agenda connects naturally with the extensive literature on these topics for nonlinear Schr\"odinger and nonlinear Klein Gordon type equations; ; see, for example, \cite{Weinstein2015,soffer-weinstein:99}.   Concerning the subtle effects on  large time weakly nonlinear dynamics, which are caused by threshold resonances, see for example
 \cite{chen-pusetari:22, schlag-luermann:23}. 
\end{enumerate}

\subsection{Notation, conventions, remarks}\label{sec:notation}
\begin{enumerate}
\item Pauli-type matrices:
\begin{equation}
  \sigma_1 = \begin{pmatrix}
    0 & 1 \\ 1 & 0
    \end{pmatrix}, \qquad \sigma_3 = \begin{pmatrix}1 & 0 \\ 0 & -1 \end{pmatrix}, \qquad 
\sigma_\star(\tau) = \begin{pmatrix}
    0 & e^{-i\tau} \\ e^{i\tau} & 0
    \end{pmatrix}.
\label{eq:Pauli}\end{equation}
\item For an operator $T:L^2(\R ;\C^2 )\to L^2(\R ;\C^2)$, which arises from an integral kernel,  we denote its kernel by $T(x,y)$, i.e., for
$f\in L^2 (\R ;\C^2)$,
$$(Tf)(x) = \int\limits_{\R} T(x,y)f(y) \, dy \, .$$
\item Let $a,b,c > 0 $ be such that $a\le b$ and $a\le c$. Then, 
\begin{align}
    a \leq \frac{2}{b^{-1}+c^{-1}}\ ,
\label{eq:abc-ineq}\end{align}
\item $\langle x \rangle = \sqrt{1+x^2}$
\item $\lambda (\omega)=\sqrt{1-\omega ^2}$ for all $\omega \in (-1,1)$, and as the unique analytic continuation of that root for all $\omega \in U$ where $U = \C \setminus \{ (-\infty ,-1] \cup [1, \infty) \}$ (see Lemma \ref{lem:lam-branch}).
\item The resolvent operator is defined by $\Reso (\omega;\tau)\equiv (\D (\tau)-\omega)^{-1} $ for every $\omega \in \C$, and for $k\in \R$ we define   $$\Resk_{\pm}(k,\tau) \equiv \Reso(\sqrt{1+k^2}\pm i0,\tau) \, .$$  
\end{enumerate}

\subsection{Structure of the article}\label{sec:outline} 
We first provide a rough outline of the proof in Section \ref{sec:strategy}. We then develop the spectral theory of the Dirac operators $\D (\tau)$: the resolvent kernel is constructed in closed form in Section \ref{sec:Resolvent1};  In Section \ref{sec:LAP}, we the limits of the resolvent kernel as the spectral parameter approaches the essential spectrum; In Section \ref{sec:scattering}, we solve the associated scattering problem and relate it to the spectral measure. Finally, the main result (Theorem \ref{maintheorem}), as well as its variation (Theorem \ref{notmaintheorem}), are proved in Section \ref{sec:mainproof}.
\subsection*{Acknowledgements} This research was supported in part by National Science Foundation grants 
DMS-1908657 and DMS-1937254 (MIW),  Simons Foundation Math + X Investigator Award \#376319 (MIW, JK, AS), the Binational Science Foundation Research Grant \#2022254 (MIW, AS), and the AMS-Simons Travel Grant (AS).
This research was initiated while JK was a postdoctoral fellow and AS was an associate research Scientist, both in the Department of Applied Physics and Applied
Mathematics at Columbia University.
%%%%%%%%%%%%%%%%%%%%%%%%%%%%%%%%%%%%%%%%%%%%%%%%%%%%%%%%%%%%%%%%%%%%%%%%%%
\section{Strategy and overview of the proof of Theorem \ref{maintheorem} \label{sec:strategy}
}
%%%%%%%%%%%%%%%%%%%%%%%%%%%%%%%%%%%%%%%%%%%%%%%%%%%%%%%%%%%%%%%%%%%%%%%%%%

For any $\alpha_0\in L^2(\R,\C)\bigcap L^1(\R,\C)$ data, by the $L^2(\R)-$ functional calculus and  Stone's formula for the spectral measure associated with $\D(\tau)$ \cite{hall2013quantum, kopylova2014dispersion, teschl2014mathematical}, we have that:
\begin{align*}
     e^{-i\D(\tau)t}P_{ac}(\D(\tau))&\langle \D(\tau)\rangle^{-3/2-\eps}\alpha_0 \\ 
     &=\frac{1}{2\pi i}\int_{\sigma_{ac}(\D(\tau))}e^{-i\omega t}\left[\Reso(\omega+i0,\tau)-\Reso(\omega-i0,\tau)\right]\alpha_0 \langle \omega\rangle^{-3/2-\eps}d\omega\ , \numberthis \label{eq:decayEst_stone}
\end{align*}
where $\Reso(\omega,\tau)$ is the resolvent
$$
    \Reso(\omega,\tau) = (\D(\tau)-\omega)^{-1} \, ,
$$
and where $\sigma_{\rm ac} = (-\infty,1] \cup [1,\infty)$ is the continuous spectrum of $\D (\tau)$. 
\begin{remark}
Since the continuous spectrum is made up of the two disjoint intervals, the integral~\eqref{eq:decayEst_stone} breaks up into two disjoint integration domains. Throughout the proof, we treat only the positive part of the spectrum; the arguments for the negative part follow analogously. In particular, treating $[1,\infty)$ independently allows us to use the change of variables \eqref{eq:cov}, which simplifies the analysis considerably. 

Alternatively, in \Cref{app:spectralsymmetry} we establish the invariance of the equation  $\D(\tau)\alpha = \omega \alpha $ under the transformation

$$\left(\omega, \, \tau,\,  \alpha (x) \right) \mapsto \left(-\omega, \, -\tau \, \, {\rm mod} (2\pi), \, \sigma_3S \alpha(-x) \right) \, .  $$ Hence, proving decay-rates for data in $P_{ac}\left(\D(\tau)\mathbbm{1}_{[1,\infty)}\right) $ for {\em } all $\tau \in [0,2\pi)$ immediately leads to the same estimates for  data in $P_{ac}\left(\D(\tau)\mathbbm{1}_{(-\infty, -1]}\right) $.
\end{remark}
Define $$P(\D(\tau)) \equiv P_{ac}\left(\D(\tau)\mathbbm{1}_{[1,\infty)}\right) \, ,$$ to be the projection onto the positive part of the continuous spectrum. Then replacing $P_{ac}$ in \eqref{eq:decayEst_stone} with $P$ yields
\begin{align*}
  \alpha(x,t)&=   e^{-i\D(\tau)t}P(\D(\tau)) \langle \D(\tau)\rangle^{-3/2-\eps}\alpha_0  \\
     &= \frac{1}{2\pi i}\int_1^\infty    e^{-i\omega t}\left[\Reso(\omega+i0,\tau)-\Reso(\omega-i0,\tau)\right]\alpha_0  \ \langle \omega\rangle^{-3/2-\eps} \,  d\omega \, . \numberthis \label{eq:stone_positive_w}
\end{align*}
Next we make the change of variables 

\begin{equation}\label{eq:cov} \omega \equiv \sqrt{1+k^2}
\end{equation}
and define
\begin{align*}
    \Resk_{\pm}(k,\tau) &\equiv \Reso(\sqrt{1+k^2}\pm i0,\tau) = (\D(\tau)-\sqrt{1+k^2}\pm i0)^{-1} \, .  
    %\numberthis \label{eq:Rktau}
\end{align*}
Then, 
\begin{align}
 \alpha(x,t) &= \frac{1}{2\pi i}\int_{0}^{\infty} dk\ \langle k\rangle^{-3/2-\eps}\frac{k}{\sqrt{1+k^2}} e^{-i\sqrt{1+k^2} t}\left[\Resk_{+}(k,\tau)-\Resk_{-}(k,\tau)\right]\alpha_0  \, ,  \label{eq:stones}
\end{align}
This last representation of the dynamics is the main object of our analysis.
\begin{remark}\label{rem:kw}
When carefully substituting $\omega = \sqrt{1+k^2}$ into \eqref{eq:stone_positive_w}, one really gets
$$
    \langle \omega \rangle = \langle \sqrt{1+k^2}\rangle 
    =(1+(\sqrt{1+k^2})^2)^{1/2}
    =(2+k^2)^{1/2} \, ,$$
in \eqref{eq:stones}, instead of $\langle k \rangle$. However, we allow this small abuse of notations in \eqref{eq:stones} since the analysis is unchanged by this difference.
\end{remark}

The strategy for proving Theorem \ref{maintheorem} from here on is as follows:

\begin{itemize}
\item In Section \ref{sec:Resolvent1}
 we construct the resolvent kernel, $\Reso_{\pm}(\omega,\tau)(x,y)$ for spectral parameter, $\omega$, in the complement of $\sigma(\D(\tau))$,  by the variation of parameters method for ODEs (Proposition \ref{prop:resolvent}).
    \item Results in Sections \ref{sec:Resolvent1}, \ref{sec:LAP}, and \ref{sec:scattering} lead to explicit expression for the solution $\alpha(t,x)$ in \eqref{eq:stones} as an oscillatory 
     integral.
   % \eqref{eq:stones} as a double integral, by computing the kernel of the resolvent operators $\Resk_{\pm}(k,\tau)$. To do so, we use ODE techniques, first to construct the kernel of the resolvent operator $\Reso (\omega,\tau)$ for $\omega\not\in~\sigma_{\rm ac} (\D (\tau))$ (Section \ref{sec:Resolvent1}), and then the limiting kernels  $\Resk_{\pm} (k,\tau)$ (Section \ref{sec:LAP}). 
  %  \item The kernel of the difference operator $\Resk_+ - \Resk_- $ is shown to be related to a transmission coefficient of the associated scattering problem (Section \ref{sec:scattering}
    \item The oscillatory integral is divided into a high-energy and low-energy components (see \eqref{eq:highlow}).
    \item The high energy component is bounded from above uniformly in $\tau$ using a stationary phase argument (Lemma \ref{lemma:vandercorput}).
    \item The low energy component is bounded from above in two steps: first a uniform $t^{-1/2}$ bound is proved (Proposition \ref{lem:t12_loweng}). Then a more careful analysis of the transmition coefficient leads to a $t^{-3/2}$ decay rate for non-resonant $\tau$.
\end{itemize}

%%%%%%%%%%%%%%%%%%%%%%%%%%%%%%%%%%%%%%%%%%%%%%%%%%%%%%%%%%%%%%%%%%%%%%%%%%
\section{The resolvent operator $\Reso(\omega;\tau)$}\label{sec:Resolvent1}
%%%%%%%%%%%%%%%%%%%%%%%%%%%%%%%%%%%%%%%%%%%%%%%%%%%%%%%%%%%%%%%%%%%%%%%%%%
In this section we construct the resolvent, $\Reso(\omega,\tau)=(\D(\tau)-\omega)^{-1}$, by providing an explicit, closed-form expression for its kernel. For any $\omega \in \C$, the resolvent operator is defined in such a way that if $\Reso(\omega;\tau)f = \beta$ for $\beta,f\in L^2(\R;\C^2)$, then
\begin{equation}\label{eq:res_ode}
  \left( \D (\tau)-\omega\right)\beta = f \,  . 
\end{equation}
Hence, the resolvent acting on $f$ returns an $L^2$ solution to the inhomogeneous ODE forced by $f$. Our first goal is to solve this ODE for $\omega$ {\em outside} of the spectrum i.e., $\omega\in U$ where
\begin{equation}\label{eq:Udef}
U\equiv\C\setminus \left\{(-\infty,-1]\cup[1,\infty)\right\} \,  .
\end{equation}
To solve the ODE \eqref{eq:res_ode} in $L^2(\R ;\C^2)$  (and hence to construct the resolvent kernel), we will use the method of variation of parameters: in Section \ref{sec:homogeneous} we first solve the corresponding homogeneous ODE $(\D (\tau)-\omega)\jostw = 0$ to find asympstotically decaying solutions, known as the {\em Jost solution} and are denoted by $\jostw_{\pm} (x;\omega, \tau)$. Afterwards, in Section \ref{sec:resolventKernel}, we use the Jost solutions to construct the kernel, see Proposition \ref{prop:resolvent}.

\subsection{The homogeneous equation}\label{sec:homogeneous}
 
We begin by considering general solutions of the  homogeneous ordinary differential equation
\begin{equation*}%\label{eq:homogenODE}
    (\D(\tau) - \omega)\beta = 0 .
\end{equation*}
This may be written as:
\begin{equation}\label{eq:beta}
    \partial _x \beta (x) = M(x;\omega,\tau)\beta(x)
    \end{equation}
    where 
    \begin{equation}\label{eq:Mxktau}
  M(x;\omega,\tau)\equiv \begin{cases} M_{+}(\omega,\tau) \, , & x>0 \, , \\
     M_-(\omega,\tau) \, , & x<0 \, ,
     \end{cases}
\end{equation}
and 
\begin{equation}\label{eq:Mpm_def}
     M_+ (\omega,\tau) \equiv \begin{pmatrix}-i\omega & ie^{-i\tau}\\ -i e^{i\tau} & i\omega
    \end{pmatrix} \, , \qquad M_- (\omega,\tau) \equiv \begin{pmatrix}
    -i\omega & i \\ -i & i\omega
    \end{pmatrix} \, .
\end{equation}
We observe that
\begin{align*}
    \det\left(M_{\pm}(\omega,\tau)\right) = \omega^2 - 1\ \quad \textrm{and}\ \quad  
    \tr\left(M_{\pm}(\omega,\tau) \right) = 0 \ \, .
\end{align*}
Hence, for a fixed $\omega$ the eigenvalues of $M_+$ and $M_-$ are solutions, $\lambda$, of 
\begin{align}\label{eq:poly}
    \lambda^2 = 1-\omega^2.
\end{align}
The following lemma ensures that, on an appropriate domain, $U$, we may choose a  solution of \eqref{eq:poly}, $\lo$, with   strictly positive real part.
\begin{lemma}\label{lem:lam-branch}
There is a holomorphic function $\omega\mapsto\lambda(\omega)$, defined on the domain  $U$ (see \eqref{eq:Udef}),
such that for all $\omega\in U$
\begin{align*}
    \lambda^2(\omega)=1-\omega^2\ \textrm{and}\ \re(\lo) > 0.
\end{align*}
Furthermore, $\overline{\lo}=\lambda(\overline\omega)$
 for all $\omega\in U$.
\end{lemma}

\begin{proof}
We claim that the desired mapping is $\omega\in U\mapsto\lambda(\omega)=\sqrt{1-\omega^2}\in\C$ (we take the positive square root) analytically continued from the interval $-1<\omega<1$ to the domain $U$. We need only check that $\re(\lambda(\omega))$ is strictly positive on $U$. If not, then by continuity of $\lambda(\omega)$ on $U$ there is some $\omega_0\in U$ where  $\re(\lambda(\omega_0))=0$. Then,
$  \lambda(\omega_0)= i\zeta$,
for some $\zeta\in\R$. Squaring both sides of this equation yields
$ 1-\omega_0^2 = -\zeta^2$
or equivalently
$ \omega_0^2 = 1+\zeta^2 > 1$, which contradicts that assumption that $\omega_0\in U$.
\end{proof}

%For $x>0$ it can be written as 
%\begin{subequations}\label{eq:M_system}
%\begin{equation}\partial _x \beta (x) = M_{+}%(k,\tau)\beta  \, ,
%\end{equation}
%and similarly for $x<0$
%\begin{equation}\partial _x \beta (x) = M_{-}%(k,\tau)\beta \, ,
%\end{equation}
%\end{subequations}
For all $\omega\in U$, the matrices $M_+(\omega,\tau)$ and $M_-(\omega,\tau)$ (see \eqref{eq:Mpm_def}) both have eigenvalues~$\pm\lo$.
  The  eigenpairs of $M_{\pm}$ are given by
\begin{subequations}\label{eq:vpmpm}
\begin{align}
    M_{-}(\omega,\tau)v_-^{(+)} &= +\lo\ v_-^{(+)} \, , &v_-^{(+)} &= \begin{pmatrix}
    1 \\ \omega-i\lo
    \end{pmatrix}  \, , \\
    M_{-}(\omega,\tau)v_-^{(-)} &= -\lo\ v_-^{(-)} \, , &v_-^{(-)} &= \begin{pmatrix}
    1 \\ \omega+i\lo
    \end{pmatrix} \, , \\
    M_+(\omega,\tau)v_+^{(+)} &= +\lo\ v_+^{(+)} \, , &v_+^{(+)}&= \begin{pmatrix}e^{-i\tau} \\ \omega-i\lo \end{pmatrix} \, ,\\
     M_+(\omega,\tau)v_+^{(-)} &= -\lo\ v_+^{(-)} \, , &v_+^{(-)} &= \begin{pmatrix}e^{-i\tau} \\ \omega+i\lo \end{pmatrix} \, .
\end{align}
\end{subequations}
Any solution of \eqref{eq:beta}, $\partial_x \beta = M \beta$, can be constructed by piecing together the  states $e^{\pm \lo x}v_-^{(\pm)}$ for $x<0$
 and $e^{\pm \lo x}v_+^{(\pm)}$ for $x>0$ across
 the equation's discontinuity at $x=0$. 

 We seek a construction of the  resolvent
  $(\D(\tau)-\omega)^{-1}$ which for all $\omega\in U$, is a bounded linear operator on $L^2 (\R ;\C^2)$. Our construction uses
   the method of {\it variation of parameters}
    \cite{coddington1955theory}, which produces 
    the resolvent kernel by appropriate  modulation of two linearly independent solutions, decaying at $x\to \pm \infty$, respectively. These are introduced in the following proposition. 
     
\begin{defn}\label{def:jost}
For any $\tau \in [0,2\pi]$ and $\omega\in U$ (see \eqref{eq:Udef}), we define the  solutions $\jostw_{\pm}(x;k,\tau)$, to be the unique solutions of
\begin{subequations}%\label{eq:jost_def_ode}
\begin{align*}
     (\D(\tau)-\omega)\jostw_{\pm}(x;\omega,\tau) = 0\ ,\quad -\infty<x<\infty,
\end{align*}
with the asymptotic behavior
\begin{align}
     \lim_{x\to+\infty} e^{\lo x}\jostw_{+}(x;\omega,\tau) &= v_+^{(-)}(\omega,\tau) \, , \label{xi+bc}\\
      \lim_{x\to-\infty} e^{-\lo x} \jostw_{-}(x;\omega,\tau) &= v_-^{(+)}(\omega,\tau) \, ,\label{xi-bc}
\end{align}
where $\lambda(\omega)$ is given by Lemma \ref{lem:lam-branch}. For $\omega\in U$, $\jostw_+(x;\omega,\tau)$
 is exponentially decaying as $x\to\infty$ 
  and $\jostw_-(x;\omega,\tau)$
 is exponentially decaying as $x\to-\infty$.
\end{subequations}
\end{defn}
\begin{remark}
In Proposition \ref{prop:xi_pm}  we display explicit expressions for the homogeneous solutions, $\jostw _{\pm}$,
 of $(\D(\tau)-\omega)\jostw=0$ for a piecewise constant domain wall. For the case of general domain walls, the existence of a unique solution satisfying
 \eqref{xi+bc} (respectively \eqref{xi-bc})  follows from a standard equivalent formulation as a fixed point problem 
 for a Volterra (integral) operator of the type that arise in the construction of {\it Jost solutions} for Schr{\"o}dinger operators \cite{reed1979scattering}. 
   %For the system of interest here,  \eqref{xi+bc} and  \eqref{xi-bc} are identities which hold, respectively, for  all $x>0$ and all $x<0$ -- not only in a limiting sense.
\end{remark}
  The solutions  $\jostw_\pm(x;\omega,\tau)$,
   introduced in Definition \ref{def:jost}, are displayed in the following proposition, whose proof is presented in Appendix \ref{app:pf_prop_xipm}.
\begin{prop}\label{prop:xi_pm}
Assume $\omega\in U = \C\setminus\{(-\infty,-1]\cup[1,\infty) \}$ so that $\re(\lo)>0$, and $\tau \in [0,2\pi]$.
\begin{enumerate}
\item 
The solutions $\jostw_\pm(x;\omega,\tau)$ have the explicit expressions:
\begin{subequations}\label{eq:xi_pm}
\begin{align} 
    \jostw_{+}(x;\omega,\tau) &= \begin{cases}  v^{(-)}_+\ e^{-\lo x} & x > 0\\
    A\ v_-^{(+)} e^{\lo x}+B\ \ v_-^{(-)} e^{-\lo x} & x < 0
    \end{cases} \label{eq:xi+}\\
     \jostw_{-}(x;\omega,\tau) &= \begin{cases} C\  v_+^{(+)} e^{\lo x}+D\ v_+^{(-)} e^{-\lo x} & x > 0\\
     v_-^{(+)}\ e^{\lo x} & x < 0
    \end{cases}.\label{eq:xi-}
\end{align}
\end{subequations}
Here, the vectors $v_\pm^{(\pm)}=v_\pm^{(\pm)}(\omega,\tau)$ are displayed in \eqref{eq:vpmpm} and $A,B,C,D$, which depend on $\omega$ and $\tau$,  are given by
\begin{subequations}\label{eq:ABCD}
\begin{align}
    A (\omega, \tau)&=  \frac{(e^{-i\tau}-1)(\omega+i\lo)}{2i\lo}, \\
    B(\omega, \tau) &=
    \frac{\omega(1-e^{-i\tau}) + i\lo(1+e^{-i\tau})}{2i\lo}  , \label{eq:b}\\
    C (\omega, \tau) &=\ \frac{\omega(e^{i\tau}-1)+i\lo(e^{i\tau}+1)}{2i\lo}, \\
    D(\omega, \tau) &= \frac{(1-e^{i\tau})(\omega-i\lo)}{2i\lo}.
\end{align}
\end{subequations}
\item Moreover, we have the algebraic relations
\begin{subequations}\label{eq:Sxi_symmetries}
\begin{align}
    \overline{\jostw_{+}(x;\omega,\tau)} &= S(\tau)\jostw_{-}(-x;\overline{\omega},\tau)\ ,\label{eq:Sxi_symmetries1} \\
    \overline{\jostw_{-}(x;\omega,\tau)} &= S(\tau)\jostw_{+}(-x;\overline{\omega},\tau)\ ,\label{eq:Sxi_symmetries2}
\end{align}
\end{subequations}
where $S(\tau)$ is the $2\times 2$ matrix defined
\begin{align*}
    S(\tau) &= \begin{pmatrix}
     e^{i\tau} & 0 \\  0 & 1
     \end{pmatrix}.
\end{align*}

\end{enumerate}
\end{prop}

%\begin{align}
%    A(k,\tau)&= \frac{\sqrt{1+k^2}(1-e^{-i\tau})+k(1+e^{-i\tau})}{2k} , \\
%    B(k,\tau)&= \frac{\sqrt{1+k^2}(-1+e^{-i\tau})+k(-1+e^{-i\tau})}{2k}\ , \\
%    C(k,\tau)&= \frac{\sqrt{1+k^2}(1-e^{i\tau})+k(-1+e^{i\tau})}{2k}= \overline{B(-\overline{k},\tau)}\ , \\
%    D(k,\tau)&= \frac{ \sqrt{1+k^2}(-1+e^{i\tau})+k(1+e^{i\tau})}{2k}= \overline{A(-\overline{k},\tau)}.
%\end{align}

%

%%%%%%%%%%%%%%%%%%%%%%%%%%%%%%%%%%%%%%%%%%%%%%%%%%%%%%%%%%%%%%%%%%%%%%%%%%
\subsection{Constructing the resolvent kernel} \label{sec:resolventKernel}

%%%%%%%%%%%%%%%%%%%%%%%%%%%%%%%%%%%%%%%%%%%%%%%%%%%%%%%%%%%%%%%%%%%%%%%%%%
In this section we give an explicit construction of the resolvent kernel in terms of the solutions $\jostw_{\pm}$ of Proposition \ref{prop:xi_pm}.
For a fixed $\tau$, we omit the $\tau$ dependence and write $\jostw_{\pm}(x;\omega,\tau)=\jostw_{\pm}(x;\omega)$ for brevity.
%The main result of this section is an explicit formula for the resolvent kernel.
%
\begin{prop}\label{prop:resolvent}
The resolvent kernel of $\Reso(\omega,\tau) = (\D(\tau)-\omega)^{-1}$ is given by
\begin{align*}
    \Reso(\omega,\tau)(x,y) &= \frac{i}{\varphi(\omega,\tau)}\begin{pmatrix} \jostw_{+,1}(x;\omega) & \jostw_{-,1}(x;\omega)\\ \jostw_{+,2}(x;\omega) & \jostw_{-,2}(x;\omega)\end{pmatrix}\begin{pmatrix}
    \mathbbm{1}_{(-\infty,x]}(y)\jostw_{-,2}(y;\omega) & \mathbbm{1}_{(-\infty,x]}(y)\jostw_{-,1}(y;\omega) \\ \mathbbm{1}_{[x,\infty)}(y)\jostw_{+,2}(y;\omega) & \mathbbm{1}_{[x,\infty)}(y)\jostw_{+,1}(y;\omega)
    \end{pmatrix}\ ,
\end{align*}
where
\begin{align}\label{eq:varphi_def}
    \varphi(\omega,\tau) :\,= i\lo(e^{-i\tau}+1)  -\omega(e^{-i\tau}-1)\ ,
\end{align}
$\lambda (\omega)$ is given by Lemma \ref{lem:lam-branch}, and  $\jostw_{\pm, j}$ for $j=1,2$ are the $j$-th coordinates of the Jost solutions, given 
%\begin{align}
%    \jostw_{\pm}(x;k,\tau) = \begin{pmatrix}
%     \jostw_{\pm,1}(x;k,\tau)\\ \jostw_{\pm,2}(x;k,\tau)
%    \end{pmatrix}\ ,
%\end{align}
which are given by Proposition \ref{prop:xi_pm}.
\end{prop}
\begin{proof}
To obtain the resolvent operator, we need to solve the inhomogeneous ODE for $\gamma (x)$
\begin{equation*}%\label{eq:inverted_resolvent}
    \left(\D(\tau)-\omega\right)\gamma (x) = f(x)\ ,
\end{equation*}
for general $f\in L^2(\R ;\C^2)$, which we do by the method of variations of parameters. This problem can be rewritten as
\begin{align}\label{eq:isigma3f_ode}
    (\partial_x-M(x,\omega))\gamma (x) = -i\sigma_3 f(x)\ ,
\end{align}
where $M(x,\omega)=M(x;\omega,\tau)$ is given by \eqref{eq:Mxktau}. Let $X(x;\omega)=X(x;\omega,\tau)$ be the fundamental matrix (of the homogeneous problem) with columns given by $\jostw_{\pm}(x;\omega)=(\jostw_{\pm,1},\jostw_{\pm,2})^\top$,
\begin{align*}
    X(x;\omega) = \begin{pmatrix}
    \jostw_{+}(x;\omega) & \jostw_{-}(x;\omega)
    \end{pmatrix}.
\end{align*}
A particular solution of the inhomogeneous system \eqref{eq:isigma3f_ode} can be found in the form
\begin{align}\label{eq:varofparam}
    \gamma(x;\omega) = X(x;\omega)c(x;\omega)\ ,
\end{align}
where  $c(\cdot;\omega):\R\to \C^2$ is to be determined.
Substitution into \eqref{eq:isigma3f_ode} yields
\begin{align}\label{eq:xcf_varparam}
    X(x;\omega)c'(x;\omega) = -i\sigma_3 f(x).
\end{align}

To invert $X(x;\omega)$ in \eqref{eq:xcf_varparam}, note that since $\tr M(x;\omega)\equiv 0$ for all $x\in \R$, we have by Liouville's formula
\begin{align*}
    \det X(x;\omega) &= e^{\int_0^x\tr M(y;\omega)\, dy }\det X(0;\omega) = \det X(0;\omega)\\
    &=\omega(e^{-i\tau}-1) -i\lo(e^{-i\tau}+1). 
\end{align*}
We set, as in \eqref{eq:varphi_def},
\begin{equation*}
  \varphi(\omega,\tau)\equiv i\lo(e^{-i\tau}+1)  -\omega(e^{-i\tau}-1).
  %\label{eq:varphi-def}
\end{equation*}
Solving \eqref{eq:xcf_varparam} for $c'(x;\omega)$, we get
\begin{align}
    c'(x;\omega) &= -i X(x;\omega)^{-1}\sigma_3 f(x) \nonumber\\
    &= \frac{i}{\varphi(\omega,\tau)}\begin{pmatrix}
    \jostw_{-,2}(x;\omega) &  \jostw_{-,1}(x;\omega) \\ -\jostw_{+,2}(x;\omega) & -\jostw_{+,1}(x;\omega)
    \end{pmatrix}\begin{pmatrix}
    f_1(x)\\ f_2(x)
    \end{pmatrix}\label{eq:cprime}
\end{align}
When constructing $c(x;\omega)$ by integration of \eqref{eq:cprime},
we must choose limits of integration of each component. 
 Informed by the exponential decay 
  of $\jostw_\pm$ as $x\to\pm\infty$, we find that in order
  to ensure that $\gamma(x;\omega)=X(x;\omega)c(x;\omega)$ decays as $x\to\pm\infty$ we must take:
  
\begin{align*}
    c(x,\omega) = \frac{i}{\varphi(\omega,\tau)}\int_{\R}\begin{pmatrix}
    \mathbbm{1}_{(-\infty,x]}(y)\jostw_{-,2}(y;\omega) & \mathbbm{1}_{(-\infty,x]}(y)\jostw_{-,1}(y;\omega) \\ \mathbbm{1}_{[x,\infty)}(y)\jostw_{+,2}(y;\omega) & \mathbbm{1}_{[x,\infty)}(y)\jostw_{+,1}(y;\omega)
    \end{pmatrix}\begin{pmatrix}
    f_1(y)\\ f_2(y)
    \end{pmatrix}dy
\end{align*}
With this choice we can write the resolvent in kernel form:
\begin{align}\label{eq:resolvent}
    \Reso(\omega,\tau)f(x) = \int_{\R}\Reso(\omega,\tau)(x,y) f(y) dy , .
\end{align}
where
\begin{align*}%\label{eq:resk}
  &  \Reso(\omega,\tau)(x,y) =\\
&\ \frac{i}{\varphi(\omega,\tau)}\ \begin{pmatrix} \jostw_{+,1}(x;\omega) & \jostw_{-,1}(x;\omega)\\ \jostw_{+,2}(x;\omega) & \jostw_{-,2}(x;\omega)\end{pmatrix}\begin{pmatrix}
    \mathbbm{1}_{(-\infty,x]}(y)\jostw_{-,2}(y;\omega) & \mathbbm{1}_{(-\infty,x]}(y)\jostw_{-,1}(y;\omega) \\ \mathbbm{1}_{[x,\infty)}(y)\jostw_{+,2}(y;\omega) & \mathbbm{1}_{[x,\infty)}(y)\jostw_{+,1}(y;\omega) \nonumber
    \end{pmatrix}.
\end{align*}
\end{proof}

Our construction of $\Reso(\omega,\tau)(x,y)$ implies the following result, which encompasses Theorem \ref{prop:specD}:
\begin{prop}\label{prop:eigs-poles}
  \begin{enumerate}\item  Let $\omega_{\tau}$ be defined as
  \begin{align*}
  \omega_{\tau} \equiv   \cos(\tau/2) \, . 
  \end{align*}
  For $\omega\in \C\setminus\{(-\infty,-1]\cup\{\omega_{\tau}\}\cup[1,\infty) \}$ , the mapping $f(x)\mapsto \Reso(\omega,\tau)f(x)$, given in \eqref{eq:resolvent}, defines a bounded operator on a dense subspace of $L^2(\R;\C^2)$. 
  \item The resolvent kernel, $\Reso(\omega,\tau)(x,y)$, has a pole  at $\omega = \omega_{\tau}$.
  \item Correspondingly, $\D(\tau)$ has an eigenvalue  $\tau\in[0,2\pi]\mapsto \omega_{\tau} =  \cos(\tau/2) $, which traverses the spectral gap $(-1,1)$.
  \end{enumerate}
\end{prop}
\begin{proof} 
The energy, $\omega$, is a pole of the resolvent kernel if and only if the Wronskian $\det X(0;\omega,\tau)=\varphi(\omega,\tau)$ vanishes. Thus, 
 %\begin{equation}\label{eq:varphi_eq0}
 $  \omega( e^{-i\tau}-1) - i\lo (e^{-i\tau}+1) = 0$,
%\end{equation}
which has the unique solution
\begin{align}\label{eq:omega_tau}
    \omega_{\tau} = \cos(\tau/2).
\end{align}
The corresponding eigenstate can be deduced as follows. Consider the expression
\eqref{eq:xi+} for the function $\eta_+(x;\omega,\tau)$, satisfying $\D(\tau)\eta_+=\omega\eta_+$. By \eqref{eq:b}
and the expression \eqref{eq:varphi_def} for $\varphi(\omega,\tau)$
    \begin{equation*}
       B = -\frac{1}{2i\lambda(\omega)} \varphi(\omega,\tau).
 %      C = e^{-i\tau}\frac{1}%{2i\lambda(\omega)}\varphi(\omega,\tau). 
     \end{equation*}
     Since $\varphi(\omega_\tau,\tau)=0$, then $B=0$ and so we have that 
\[\jostw_{+}(x;\omega_\tau,\tau) = \begin{cases}  v^{(-)}_+(\omega_\tau,\tau)\times e^{-\lambda(\omega_\tau) x} & x > 0\\
    A(\omega_\tau,\tau)\ v_-^{(+)}(\omega_\tau,\tau)\times e^{\lambda(\omega_\tau) x}& x < 0
    \end{cases}, \label{eq:tau-mode}
    \]
    which decays exponentially as $x\to\pm\infty$ since $\lambda(\omega_\tau)>0$ (Lemma \ref{lem:lam-branch}). Finally, define 
    $\psi_\tau(x)$ to be a constant multiple 
     of $\jostw_{+}(x;\omega_\tau,\tau)$ for which $\|\psi_\tau\|_{L^2(\R)}=1$. This yields the expression displayed in~\eqref{eq:alpha*}. 
\end{proof}

%%%%%%%%%%%%%%%%%%%%%%%%%%%%%%%%%%%%%%%%%%%%%%%%%%%%%%%%%%%%%%%%%%%%%%%%%%
\section{Limits of the resolvent kernel as the energy approaches $\sigma_{\rm ess}(\D(\tau))$}\label{sec:LAP}

%%%%%%%%%%%%%%%%%%%%%%%%%%%%%%%%%%%%%%%%%%%%%%%%%%%%%%%%%%%%%%%%%%%%%%%%%%

A representation of the time-evolution for the Hamiltonian, $\D(\tau)$, is given in \eqref{eq:stones} based on the functional calculus and Stone's formula for the spectral measure of $\D(\tau)$. To work with this representation we need to evaluate the $\Resk_+(k,\tau)-\Resk_-(k,\tau)  $, where we recall the definition
\begin{align}
    \Resk_{\pm}(k,\tau) :\,= \lim_{\eps\to 0^+} \Reso(\sqrt{1+k^2}\pm i\eps, \tau) = \lim_{\eps\to 0^{+}}(\D(\tau)-(\sqrt{1+k^2}\pm i\eps))^{-1} ,
\label{eq:cRpm}\end{align}
for $k\in \R$ (corresponding to the spectral parameter, $\omega$, being in the essential spectrum of $\D(\tau)$). In this section we shall study the limits \eqref{eq:cRpm} by studying the resolvent kernel. Then, in Section \ref{sec:scattering} we study the properties of the difference $\Resk_+(k,\tau)-\Resk_-(k,\tau)  $, which are summarized in Proposition~\ref{prop:limAbs}.

In order to calculate the limits \eqref{eq:cRpm}, we remark that by our choice of $\lambda(\omega)$ (Lemma \ref{lem:lam-branch}) for $\omega\in[1,\infty)$ we have
\begin{align} \label{eq:squarerootlimit}
    \lim_{\eps\to 0^{+}}\lambda(\omega\pm i\eps) = \mp i\sqrt{\omega^2-1} .
\end{align}

Let $k\in\R$ and define \textit{Jost solutions}, $\jostk_{\pm}(x;k,\tau)$, which satisfy
\begin{align*}
    (\D(\tau)-\sqrt{1+k^2})\jostk_{\pm}(x;k,\tau) = 0\ ,\quad -\infty<x<\infty,
\end{align*}
with asymptotic behavior
\begin{align*}
     \lim_{x\to+\infty} e^{-ikx}\jostk_{+}(x;k,\tau) &= \begin{pmatrix}
    e^{-i\tau}\\ \sqrt{1+k^2} +k
    \end{pmatrix} \, , \\
      \lim_{x\to-\infty} e^{+ikx} \jostk_{-}(x;k,\tau) &= \begin{pmatrix}
      1 \\ \sqrt{1+k^2} - k
      \end{pmatrix}\  .
\end{align*}
The following result states that the Jost solutions arise as limits of the decaying solutions $\jostw_{\pm}(x;\omega,\tau)$ for $\omega$ approaching the essential spectrum. 
\begin{prop}\label{prop:jostkdef}
    The Jost solutions $\jostk_{\pm}(x;k,\tau)$ are related to the decaying solutions $\jostw_{\pm}(x;\omega,\tau)$ through the following limits.
    \begin{align*}
    \jostk_{+}(x;\pm \vert k\vert,\tau) &= \lim_{\eps\to 0^+}\jostw_{+}(x;\sqrt{1+k^2}\pm i\eps,\tau) ,\\
    \jostk_{-}(x;\pm \vert k\vert,\tau) &= \lim_{\eps\to 0^+}\jostw_{-}(x;\sqrt{1+k^2}\pm i\eps,\tau) .
\end{align*}
\end{prop}
In order to prove the proposition, one uses the identity
\begin{align*}
      \lim_{\eps\to 0^{+}}\lambda(\sqrt{1+k^2}\pm i\eps) = \mp i\vert k\vert ,
\end{align*}
which follows from \eqref{eq:squarerootlimit} and the substitution $\omega=\sqrt{1+k^2}$. Henceforth we assume that $k\geq 0$. An immediate consequence of this is that the limiting resolvents $\Resk_{\pm}(k,\tau)$ may be expressed in terms of the Jost solutions $\jostk_{\pm}(x;k,\tau)$
 \begin{subequations}\label{eq:R_pm_exp}
\begin{align}
    \Resk_{\pm}(k,\tau)(x,y) &= \begin{cases}
    \frac{-i}{\varphi(\pm k,\tau)}\begin{pmatrix} \xi_{-,1}(x; \pm k,\tau)\xi_{+,2}(y; \pm k,\tau) & \xi_{-,1}(x; \pm k,\tau)\xi_{+,1}(y; \pm k,\tau)\\ \xi_{-,2}(x; \pm k,\tau)\xi_{+,2}(y; \pm k,\tau) & \xi_{-,2}(x; \pm k,\tau)\xi_{+,1}(y; \pm k,\tau)\end{pmatrix} & x < y\\
    \frac{-i}{\varphi(\pm k,\tau)}\begin{pmatrix} \xi_{+,1}(x; \pm k,\tau)\xi_{-,2}(y; \pm k,\tau) & \xi_{+,1}(x; \pm k,\tau)\xi_{-,1}(y; \pm k,\tau)\\ \xi_{+,2}(x; \pm k,\tau)\xi_{-,2}(y; \pm k,\tau) & \xi_{+,2}(x; \pm k,\tau)\xi_{-,1}(y; \pm k,\tau)\end{pmatrix} & x > y
    \end{cases}\\
    &=\begin{cases} \frac{-i}{\varphi(\pm k,\tau)}\xi_{-}(x; \pm k,\tau)\xi_{+}(y; \pm k,\tau)^\top\sigma_1 & x < y\\
    \frac{-i}{\varphi(\pm k,\tau)}\xi_{+}(x; \pm k,\tau)\xi_{-}(y; \pm k,\tau)^\top\sigma_1 & x > y
    \end{cases}
\end{align}
\end{subequations}
For explicitly formulae of the resolvent kernel, see Appendix \ref{ap:resolvent_explicit}.
%%%%%%%%%%%%%%%%%%%%%%%%%%%%%%%%%%%%%%%%%%%%%%%%%%%%%%%%%%%%%%%%%%%%%%%%%%
\section{Scattering Theory}\label{sec:scattering}
%%%%%%%%%%%%%%%%%%%%%%%%%%%%%%%%%%%%%%%%%%%%%%%%%%%%%%%%%%%%%%%%%%%%%%%%%%

In this section we construct the scattering theory for $\D(\tau)$ and find explicit formulas for the transmission and reflection coefficients. For a fixed $\tau$, These coefficients relate the Jost solutions $\xi _{\pm} (x;k)$ to $\xi _{\pm} (x;-k)$, in the following way:
\begin{lemma}\label{lem:RT}
Let $k \in \R \setminus \{0\}$ be such that $\varphi(k,\tau)\neq 0$ (see \eqref{eq:varphi_def}). Then, there exist unique coefficients $T_1(k,\tau), R_1(k,\tau), T_2(k,\tau)$, and $R_2(k,\tau)$ such that 
\begin{subequations}\label{eq:algebraicrelations}
\begin{align} 
    \xi_{+}(x;-k,\tau) &= T_1(k,\tau)\xi_{-}(x;k,\tau)-R_1(k,\tau)\xi_{+}(x;k,\tau)\ , \\
    \xi_{-}(x;-k,\tau) &= T_2(k,\tau)\xi_{+}(x;k,\tau)-R_2(k,\tau)\xi_{-}(x;k,\tau).
\end{align}
\end{subequations}
We call the $T_j$'s and the $R_j$'s the transmission and reflection coefficients, respectively.
\end{lemma}
%We comment that those values for which $\varphi(k,\tau) = 0$ correspond to the (single) eigenvalue of $\D (\tau)$.
\begin{proof}
The existence and uniqueness of the coefficients $T_{1},T_{2}, R_{1},R_{2}$ follow immediately from two facts
\begin{enumerate}
    \item $\xi_{\pm}(x;\pm k,\tau)$ solve the same differential equation as $\xi_{\pm}(x;\pm k,\tau)$.
    \item For $k\neq 0$, the Jost solutions $\jostk_{+}(x;k,\tau)$ and $\jostk_{-}(x;k,\tau)$ are linearly independent.
\end{enumerate}
The first is clear as the equation is preserved with $k\to -k$. One may observe the second by a direct calculation of the Wronskian at $x=0$:
\begin{align*}
    W[\jostk_{+}(x;k,\tau),\jostk_{-}(x;k,\tau)] &=\varphi(k,\tau)\\
    &=k(e^{-i\tau}+1) - \sqrt{1+k^2}(e^{-i\tau}-1)\neq 0 .
\end{align*}

\end{proof}
We can even go further and solve for the coefficients:
\begin{prop}
For every $k\in \R \setminus \{0\}$ and $\tau \in (0,2\pi)$,
\begin{subequations}\label{eq:TjRj_relations}
\begin{align}
    T(k,\tau)  : \, = T_1(k,\tau)&= e^{-i\tau}T_2(k,\tau) \ ,\\
    \frac{R_1(k,\tau)}{T_1(k,\tau)}&= \frac{-1}{e^{-i\tau}}\frac{R_2(-k,\tau)}{T_2(-k,\tau)}\ ,
\end{align}
\end{subequations}
where 
$$  T(k,\tau)  = \frac{2k}{k(e^{i\tau}+1)+\sqrt{1+k^2}(e^{i\tau}-1)}=\frac{2e^{-i\tau}}{\varphi(k,\tau)}\, .
  $$
\end{prop}
\begin{proof}
 Applying Cramer's rule to the definition of the transmition and reflection coefficients, and having already computed the relevant Wronskians in the proof of Lemma \ref{lem:RT}, we have
\begin{align*}
    T_1(k,\tau) &= \frac{W[\xi_{+}(x;-k,\tau),\xi_{+}(x;k,\tau)]}{W[\xi_{-}(x;k,\tau),\xi_{+}(x;k,\tau)]}=\frac{2e^{-i\tau}k}{\varphi(k,\tau)} \ , \\
    R_1(k,\tau) &= -\frac{W[\xi_{-}(x;k,\tau),\xi_{+}(x;-k,\tau)]}{W[\xi_{-}(x;k,\tau),\xi_{+}(x;k,\tau)]} = \frac{W[\xi_{+}(x;-k,\tau),\xi_{-}(x;k,\tau)]}{\varphi(k,\tau)} \ , \\
    T_2(k,\tau) &= \frac{W[\xi_{-}(x;-k,\tau),\xi_{-}(x;k,\tau)]}{W[\xi_{+}(x;k,\tau),\xi_{-}(x;k,\tau)]}=\frac{2k}{\varphi(k,\tau)}\ , \\
    R_2(k,\tau) &= -\frac{W[\xi_{+}(x;k,\tau),\xi_{-}(x;-k,\tau)]}{W[\xi_{+}(x;k,\tau),\xi_{-}(x;k,\tau)]}=\frac{W[\xi_{+}(x;k,\tau),\xi_{-}(x;-k,\tau)]}{\varphi(k,\tau)}.
\end{align*}
This leads to the relations \eqref{eq:TjRj_relations}. 

The explicit formula for the transmission coefficient is given by
\begin{align*}
    T(k,\tau)
    %&= \frac{2ke^{-i\tau}}{\varphi(k,\tau)}\\
    %&= \frac{2ke^{-i\tau}}{k(1+e^{-i\tau})+\sqrt{1+k^2}(1-e^{-i\tau})}\\
    &= \frac{2k}{k(e^{i\tau}+1)+\sqrt{1+k^2}(e^{i\tau}-1)}.
\end{align*}
\end{proof}
Moreover, the square magnitude and its derivatives are given by
\begin{align}
    \vert T(k,\tau)\vert^2 &= \frac{k^2}{k^2+\sin^2(\tau/2)}\ , \label{eq:tk}\\
    \partial_k\vert T(k,\tau)\vert^2  &= \frac{2k\sin^2(\tau/2)}{(k^2+\sin^2(\tau/2))^2}\ , \label{eq:dtk}\\
      \partial_k^2\vert T(k,\tau)\vert^2  &= \frac{2\sin^2(\tau/2)(\sin^2(\tau/2)-3k^2)}{(k^2+\sin^2(\tau/2))^3}. \label{eq:d2tk}
\end{align}
From this we have the estimates
\begin{prop}\label{lem:TransmissionBounds}
We have the following bounds on $\partial_k^j\vert T(k,\tau)\vert^2$ for $j=0,1,2$.
\begin{subequations}
    \begin{align*}
    \vert T(k,\tau)\vert^2  &\leq \min\left(1,\frac{k^2}{k^2+\sin^2(\tau/2)}\right)\ , \\
    \partial_k\vert T(k,\tau)\vert^2  &\leq \min\left(\frac{1}{k},\frac{2k}{k^2+\sin^2(\tau/2)},\frac{2}{k^3}\right)\ , \\
    \partial_k^2\vert T(k,\tau)\vert^2  &\leq \min\left(\frac{3}{k^2},\frac{8}{k^2+\sin^2(\tau/2)},\frac{7}{k^4}\right).
\end{align*}
\end{subequations}
\end{prop}
\begin{proof}
    By direct computation and elementary bounds $\sin^2 (\theta) \in [0,1]$ and $k^2, k^4 \geq 0$. 
\end{proof}
% \begin{lemma}
% The transmission coefficient $T(k,\tau)$ has the following asymptotic as $k\to 0$.
% \begin{align}
%     T(k,\tau) = 1 - \theta(\vert e^{i\tau}-1\vert)(1 + \frac{2k}{e^{i\tau}-1} + o(k))\ ,
% \end{align}
% where $\theta$ is the Heaviside step function defined
% \begin{align}
%     \theta(x) = \begin{cases} 1 & x > 0\\
%     0 & x \leq 0
%     \end{cases}.
% \end{align}
% In particular
% \begin{align}
%     T(0,\tau) = \begin{cases} 1 & e^{i\tau} = 1\\
%     0 & e^{i\tau} \neq 1
%     \end{cases}.
% \end{align}
% \end{lemma}
We now state a proposition which gives several alternative ways to calculate the difference across the continuous spectrum of the two operators defined by the limiting absorption principle above.
\begin{prop}[Jump in the resolvent kernel across the real axis]\label{prop:limAbs}
For any $k\in\R$
\begin{align*}
   & \Resk_+(k,\tau)(x,y)-\Resk_-(k,\tau)(x,y)\\
    %  =&\begin{cases}\frac{\vert T(k,\tau)\vert^2}{2i e^{-i\tau}k}\left([\xi_{+}(x;k,\tau),0][\xi_{+}(y;-k,\tau),0]^\top\sigma_1+e^{-i\tau}[\xi_{-}(x;k,\tau),0][\xi_{-}(y;-k,\tau),0]^\top\sigma_1\right) & x > y\\
    % \frac{\vert T(k)\vert^2}{2ik}\left(-[0,\xi_{-}(x,k)][0,\xi_{-}(y,-k)]^\top\sigma_1+[0,\xi_{+}(x,k)][0,\xi_{+}(y,-k)]^\top\sigma_1\right) & x < y
    % \end{cases}\\
    % =&\frac{\vert T(k,\tau)\vert^2}{2ie^{-i\tau}k}\begin{pmatrix}\xi_{+,1}(x;k,\tau)\xi_{+,2}(y;-k,\tau)& \xi_{+,1}(x;k,\tau)\xi_{+,1}(y;-k,\tau)\\
    % \xi_{+,2}(x;k,\tau)\xi_{+,2}(y;-k,\tau)& \xi_{+,2}(x;k,\tau)\xi_{+,1}(y;-k,\tau)\end{pmatrix}\\
    % +&\frac{\vert T(k,\tau)\vert^2}{2ie^{-i\tau}k}\begin{pmatrix}e^{-i\tau}\xi_{-,1}(x;k,\tau)\xi_{-,2}(y;-k,\tau) & e^{-i\tau}\xi_{-,1}(x;k,\tau)\xi_{-,1}(y;-k,\tau)\\
    % e^{-i\tau}\xi_{-,2}(x;k,\tau)\xi_{-,2}(y;-k,\tau ) &     e^{-i\tau}\xi_{-,2}(x;k,\tau)\xi_{-,1}(y;-k,\tau)\end{pmatrix}
   &\qquad  =\frac{\vert T(k,\tau)\vert^2}{2i e^{-i\tau}k}\left(\xi_{+}(x;k,\tau)\xi_{+}(y;-k,\tau)^\top\sigma_1+e^{-i\tau}\xi_{-}(x;k,\tau)\xi_{-}(y;-k,\tau)^\top\sigma_1\right)\\
\end{align*}
\end{prop}
\begin{proof}
We start with $x > y$. Using \eqref{eq:algebraicrelations} along with the properties of the transmission and reflection coefficients we find
\begin{align*}
    &\Resk_+(k,\tau)(x,y)-\Resk_-(k,\tau)(x,y)\\
    =& \frac{-i}{\varphi(k,\tau)}\xi_{+}(x;k,\tau)\xi_{-}(y;k,\tau)^\top \sigma_1-\frac{-i}{\varphi(-k,\tau)}\xi_{+}(x;-k,\tau)\xi_{-}(y;-k,\tau)^\top \sigma_1\\
    =&\frac{-iT(k,\tau)}{2e^{-i\tau}k}\xi_{+}(x;k,\tau)\left(e^{i\tau}T(-k,\tau)\xi_{+}(x;-k,\tau))-R_2(-k,\tau)\xi_{-}(x;-k,\tau)\right)^\top\sigma_1\\
    -&\frac{iT(-k,\tau)}{2e^{-i\tau}k}\left(T(k,\tau)\xi_{-}(x;k,\tau)-R_1(k,\tau)\xi_{+}(x;k,\tau)\right)\xi_-(y,-k)^\top\sigma_1\\
    =&\frac{\vert T(k,\tau)\vert^2}{2i e^{-i\tau}k}\left(\xi_{+}(x;k,\tau)\xi_{+}(y;-k,\tau)^\top\sigma_1+e^{-i\tau}\xi_{-}(x;k,\tau)\xi_{-}(y;-k,\tau)^\top\sigma_1\right)
    % =&\frac{\vert T(k,\tau)\vert^2}{2ie^{-i\tau}k}\begin{pmatrix}\xi_{+,1}(x;k,\tau)\xi_{+,2}(y;-k,\tau)& \xi_{+,1}(x;k,\tau)\xi_{+,1}(y;-k,\tau)\\
    % \xi_{+,2}(x;k,\tau)\xi_{+,2}(y;-k,\tau)& \xi_{+,2}(x;k,\tau)\xi_{+,1}(y;-k,\tau)\end{pmatrix}\\
    % +&\frac{\vert T(k,\tau)\vert^2}{2ie^{-i\tau}k}\begin{pmatrix}e^{-i\tau}\xi_{-,1}(x;k,\tau)\xi_{-,2}(y;-k,\tau) & e^{-i\tau}\xi_{-,1}(x;k,\tau)\xi_{-,1}(y;-k,\tau)\\
    % e^{-i\tau}\xi_{-,2}(x;k,\tau)\xi_{-,2}(y;-k,\tau ) &     e^{-i\tau}\xi_{-,2}(x;k,\tau)\xi_{-,1}(y;-k,\tau)\end{pmatrix}
\end{align*}
The case where $x < y$ is similar. 
% \begin{align*}
%     &\Resk_+(k,\tau)(x,y)-\Resk_-(k,\tau)(x,y)\\
%     =&\frac{-i}{\varphi( k,\tau)}[0,\xi_{-}(x;k,\tau)][0,\xi_{+}(y;k,\tau)]^\top\sigma_1-\frac{-i}{\varphi(- k,\tau)}[0,\xi_{-}(x,- k,\tau)][0,\xi_{+}(y,- k,\tau)]^\top\sigma_1\\
%     =&\frac{-iT(k,\tau)}{2e^{-i\tau}k}[0,\xi_{-}(x;k,\tau)][0,T(-k,\tau)\xi_{-}(x;-k,\tau)-R_1(-k,\tau)\xi_{+}(x;-k,\tau)]^\top\sigma_1 \\
%     -&\frac{iT(-k,\tau)}{2^{-i\tau}k}[0,e^{i\tau}T(k,\tau)\xi_{+}(x;k,\tau))-R_2(k,\tau)\xi_{-}(x;k,\tau)][0,\xi_{+}(y;-k,\tau)]^\top\sigma_1 \\
%     =&\frac{\vert T(k,\tau)\vert^2}{2ie^{-i\tau}k}\left(e^{-i\tau}[0,\xi_{-}(x;k,\tau)][0,\xi_{-}(y;-k,\tau)]^\top\sigma_1+[0,\xi_{+}(x;k,\tau)][0,\xi_{+}(y;-k,\tau)]^\top\sigma_1\right)\\
%     =&\frac{\vert T(k,\tau)\vert^2}{2ie^{-i\tau}k}\begin{pmatrix}\xi_{+,1}(x;k,\tau)\xi_{+,2}(y;-k,\tau)& \xi_{+,1}(x;k,\tau)\xi_{+,1}(y;-k,\tau)\\
%     \xi_{+,2}(x;k,\tau)\xi_{+,2}(y;-k,\tau)& \xi_{+,2}(x;k,\tau)\xi_{+,1}(y;-k,\tau)\end{pmatrix}\\
%     +&\frac{\vert T(k,\tau)\vert^2}{2ie^{-i\tau}k}\begin{pmatrix}e^{-i\tau}\xi_{-,1}(x;k,\tau)\xi_{-,2}(y;-k,\tau) & e^{-i\tau}\xi_{-,1}(x;k,\tau)\xi_{-,1}(y;-k,\tau)\\
%     e^{-i\tau}\xi_{-,2}(x;k,\tau)\xi_{-,2}(y;-k,\tau ) &     e^{-i\tau}\xi_{-,2}(x;k,\tau)\xi_{-,1}(y;-k,\tau)\end{pmatrix}.
% \end{align*}
\end{proof}

Now that we have developed the requisite scattering theory we may prove the main theorem. 

%%%%%%%%%%%%%%%%%%%%%%%%%%%%%%%%%%%%%%%%%%%%%%%%%%%%%%%%%%%%%%%%%%%%%%%%%%
\section{Proof of  Theorem \ref{maintheorem}, the main theorem}\label{sec:mainproof}
%%%%%%%%%%%%%%%%%%%%%%%%%%%%%%%%%%%%%%%%%%%%%%%%%%%%%%%%%%%%%%%%%%%%%%%%%
We begin by remarking that it suffices to prove the following weaker bound which, in contrast to \eqref{eq:decayest1}
(equivalently \eqref{eq:decayest2}), is singular as $t\downarrow0$:
\begin{equation}\label{eq:decayest3}
    \|\langle x\rangle^{-2}e^{-i\D(\tau)t}P_{ac}(\D(\tau))\alpha_0\|_{L^\infty(\R)}\leq C_\epsilon\frac{1}{t^{1/2}}\frac{1}{1+\sin^2(\tau/2) t}
    \|\langle x\rangle^{2} \langle \D(\tau)\rangle^{3/2+\epsilon}\alpha_0\|_{L^1(\R)} .
\end{equation}
Indeed, from the expression for the  resolvent kernel in Appendix \ref{ap:resolvent_explicit}, we  note that for all $x$ and $y$, the resolvent kernel satisfies: $|k\Resk_{\pm}(k,\tau)(x,y)|\le C_\tau$, for some constant $C_\tau$. Therefore, directly from the representation of $\alpha(t,x)$ as an oscillatory integral of the data $\alpha_0$ in \eqref{eq:stones}, we have:
 \begin{equation}
 \|\alpha(t,x)\|_{L^\infty(\R_x)}\lesssim \|\alpha_0(y)\|_{L^1(\R_y)}\quad \textrm{for all $t\in\R$}.
\label{eq:linfinty} \end{equation}
We next derive an estimate reflecting \eqref{eq:decayest3} for $t$ bounded away from zero and  \eqref{eq:linfinty} for $t$ near zero by interpolation. Denote $$a=\|\langle x\rangle^{-2}e^{-i\D(\tau)t}P_{ac}(\D(\tau))f\|_{L^\infty(\R)} \, , \qquad b\sim
    \|\langle x\rangle^{2} \langle \D(\tau)\rangle^{3/2+\epsilon}f\|_{L^1(\R)} 
 \, , $$
 $$c\sim \frac{1}{t^{1/2}}\frac{1}{1+\sin^2(\tau/2) t}
    \|\langle x\rangle^{2} \langle \D(\tau)\rangle^{3/2+\epsilon}f\|_{L^1(\R)} 
 \, .$$ Then, using the bounds 
 \eqref{eq:decayest3} and \eqref{eq:linfinty}, we have that $a\leq b$ and $a\leq c$. A bound which is non-singular as $t\to0$ now follows from the algebraic interpolation inequality \eqref{eq:abc-ineq}.
We therefore focus on proving the decay-estimate \eqref{eq:decayest3}.

We now embark on the proof. Our strategy is to decompose the representation of the evolution operator, given in \eqref{eq:stones}, into   ``high $k$'' (or ``high energy'') and ``low $k$'' (or ``low energy'') integrals. To do so, fix $k_0 > 0$ and define a smooth cutoff function $\chi$, defined on $\R$, such that
\begin{equation}\label{eq:cutoff}
    \chi(k) = \begin{cases}
    1 & \textrm{if}\ k < k_0 \\
    0 & \textrm{if}\ k > 2k_0\ .
    \end{cases}
\end{equation}
Then, for every $\alpha_0 \in L^2 (\R ;\C ^2)$, we have from \eqref{eq:stones}) that

\begin{align*}
 \alpha(t,x) =  e^{-i\D(\tau)t} P(\D(\tau))\langle \D(\tau)\rangle^{-3/2-\eps} \alpha _0 &:= \alpha_l(t,x) + \alpha_h(t,x)\ \, , \numberthis \label{eq:highlow}
\end{align*}
where
\begin{align}
   \alpha_l(t,x) &:=\frac{1}{2\pi i}\int_{0}^{\infty}e^{-i\sqrt{1+k^2} t}\chi(k)\left[\left(\Resk_{+}(k,\tau)-\Resk_-(k,\tau)\right)\alpha_0\right] (x)\, \langle k\rangle^{-3/2-\eps}\frac{k}{\sqrt{1+k^2}} \, dk \label{eq:low} \, ,\\
   \alpha_h(t,x) &:=\frac{1}{2\pi i}\int_{0}^{\infty}e^{-i\sqrt{1+k^2} t}(1-\chi(k))\left[\left(\Resk_{+}(k,\tau)-\Resk_-(k,\tau)\right)\alpha_0\right] (x)\, \langle k\rangle^{-3/2-\eps}\frac{k}{\sqrt{1+k^2}} \, dk \, . \label{eq:high}
\end{align}

Note that $\alpha(t,x)$ in \eqref{eq:highlow} is the solution of the time-dependent Dirac equation \eqref{eq:Diraceq} with the smoothed initial data $P(\D(\tau))\langle \D(\tau)\rangle^{-3/2-\eps} \alpha _0$. The functions $\alpha_l(t,x)$ and $\alpha_l(t,x)$ are, respectively, the low $k$ and high $k$ components of $\alpha(t,x)$. We estimate these separately in the following sections.
% i.e., we need to provide upper bounds of the form 
% $$ \|\langle x \rangle^{-2} \alpha_j (t,x)\langle x\rangle^{-2} \|_{L^{\infty}} \leq C(t) \|\alpha _j \|_{L^1} \, , \qquad j=h,l \, ,$$
% for some $C:\R _{\geq 0} \to \R _{\geq 0}$ with $C(t)\to 0$ as $t\to \infty$.

\subsection{High Energy Estimates}\label{sec:high_energy}
In this section we prove high-energy the following $\tau$-independent bounds on $\alpha_h(t,x)$, defined in \eqref{eq:highlow}.
\begin{theorem} \label{thm:high_energy}
 There exist positive constants $C_{h,1}$ and $C_{h,2}$, such that for any $\tau\in[0,2\pi)$ and $t\ge0$
 \begin{align}
     \|\langle x\rangle^{-1}\alpha_{h}(t,x)\|_{L^{\infty}(\R_x)}&\leq C_{h,1}\frac{1}{| t|^{3/2}}\|\langle y\rangle \alpha_0(y)\|_{L^{1}(\R_y)}\label{eq:hi-wt} \\
     \|\alpha_{h}(t,x)\|_{L^{\infty}(\R_x)}&\leq C_{h,2}
     \frac{1}{|t|^{1/2}}\|\alpha_0(y)\|_{L^1(\R_y)}\label{eq:hi-no-wt}
 \end{align}
\end{theorem}
\begin{proof}
To bound $\alpha_h$, the high energy part of $\alpha$, it is sufficient to bound the contributions from $\mathcal{R}_+$ and $\mathcal{R}_-$ (see \eqref{eq:cRpm})  individually.
Write
\begin{subequations}\label{eq:alpha_hpm}
\begin{equation}
  \alpha_h(t,x) =  \alpha_h^{(+)}(t,x)-\alpha_h^{(-)}(t,x)
  \end{equation}
  where
  \begin{equation}
  \alpha_h^{(\pm)}(t,x) \equiv  \frac{1}{2\pi i}\int_{0}^{\infty}e^{-i\sqrt{1+k^2} t}(1-\chi(k))\left[\mathcal{R}_{\pm}(k,\tau)\alpha_0 \right](x) \, \langle k\rangle^{-3/2-\eps}\frac{k}{\sqrt{1+k^2}} \, dk  .
  \end{equation}\label{eq:ahpm_def}
\end{subequations}

We will estimate $\alpha_h^{(+)}$ explicitly; the estimates for $\alpha_h^{(-)}$ follow analogously. The  operator $\mathcal{R}_+ (k,\tau)$, see \eqref{eq:cRpm},  acting on  functions  $\beta \in L^2 (\R;\C^2)$, can be written via its kernel representation, i.e.,
$$[\mathcal{R}_+(k,\tau)\beta ](x) = \int\limits_{\R} \mathcal{R}_+(k,\tau)(x,y)\beta(y) \, dy\ ,$$
where the resolvent kernel $\mathcal{R}_+(k,\tau)(x,y)$, is displayed in  \eqref{eq:ahpm_def}. Substitution into 
\eqref{eq:ahpm_def} yields, via the Fubini-Tonelli theorem, that for any $x\in \R$
\begin{equation}
\alpha_h^{(+)}(t, x) = \int_{\R} A_t(x,y) \alpha_0(y) dy,
\end{equation}
where 
\[ A_t(x,y) \equiv \frac{1}{2\pi i}\int_{0}^{\infty}e^{-i\sqrt{1+k^2} t}(1-\chi(k))\Resk_{+}(k,\tau)(x,y) \, \langle k\rangle^{-3/2-\eps}\frac{k}{\sqrt{1+k^2}}dk\]
 Hence, 
\begin{equation}\label{eq:h_sup_prebound}
  \|\alpha_{h}(t,x)\|_{L^{\infty}(\R_x)}  \le \sup_{z,z'\in\R} |A_t(z,z')|\times  \|\alpha_0(y)\|_{L^1(\R_y)} \, ,
    \end{equation}
    and (by inserting a $1=\langle y \rangle ^{-1} \langle y \rangle$ terms into the integrand)
 \begin{equation}\label{eq:h_sup_prebound1}
 \|\langle x\rangle^{-1}\alpha_{h}(t,x)\|_{L^{\infty}(\R_x)}  \le \sup_{z,z'\in\R} \left(\langle z\rangle^{-1}|A_t(z,z')|\langle z'\rangle^{-1}\right)\times  \|\langle y\rangle\alpha_0\|_{L^1(\R_y)} \, .
    \end{equation}   
     Theorem \ref{thm:high_energy} will follow if we can establish the following kernel bounds:
    \begin{align}
    \sup_{x,y\in\R} |A_t(x,y)| &\le C \vert t\vert^{-\frac12} \label{eq:kbound1}\\
    \sup_{x,y\in\R} \ \langle x\rangle^{-1}|A_t(x,y)|\langle y\rangle^{-1} &\le C'\vert t\vert^{-\frac32}
   \label{eq:kbound2} \end{align}
   We focus on the proof of bound \eqref{eq:kbound2} and its use in the proof of \eqref{eq:hi-wt}.  The proof of \eqref{eq:kbound1} and its application to the proof of \eqref{eq:hi-no-wt} is simpler.   
   
    $A_t(x,y)$ is defined as an integral with respect to $k\in{\rm supp}(1-\chi)=[k_0,\infty)$. To bound $A_t(x,y)$,  we express $[k_0,\infty)$
as union of overlapping intervals of exponentially increasing size: $[k_0,\infty)=\bigcup_{j\ge0}[k_0 2^j,k_0 2^{j+2}],\ j\ge0$, and use the corresponding smooth partition of unity
 \begin{equation}\label{eq:p_of_unity}
 \mathbbm{1}_{(k_0, \infty)}(x) = \sum_{j\ge0} \chi_j(x),\quad 
 \chi_j  \in C_c^{\infty} \left( [k_0 2^j,k_0 2^{j+2}]\right).
 \end{equation}
By the triangle inequality, for every $x,y\in \R$
\begin{align}\label{eq:Ij_breakup}
   \left\vert A_t(x,y)\right\vert \leq \sum_{j=0}^{\infty} I_j(x,y)\ ,
\end{align}
where
\begin{align}\label{eq:Ij_def}
    I_j (x,y) =  \left\vert \frac{1}{2\pi i}\int_{k_0}^{\infty}e^{-it\sqrt{1+k^2}}\mathcal{R}_+(k,\tau)(x,y)\langle k\rangle^{-3/2-\eps} \chi_j(k)\frac{k }{\sqrt{1+k^2} } \, dk\right\vert.
\end{align}

The kernel, $\mathcal{R}_{+}(k,\tau)(x,y)$, is a $2\times 2$ matrix whose elements are displayed in Appendix~\ref{ap:resolvent_explicit}. Based on the precise expressions for these matrix elements, $I_j$ can be written as a linear combination of expressions of the following type:
\begin{align}\label{eq:gen-form}
    \left\vert \frac{1}{2\pi i}\int_{k_0}^{\infty}e^{-it\sqrt{1+k^2}\pm ik( x\pm y)}L(k)\langle k\rangle^{-3/2-\eps}\chi_j(k)\frac{k }{\sqrt{1+k^2}} \, dk \right\vert\ ,
\end{align}
where $L(k)$ is a rational function of  $k$ and $\sqrt{1+k^2}$, which can be read off the formulae in Appendix~\ref{ap:resolvent_explicit}. To obtain  upper bounds on \eqref{eq:gen-form} we make use of the following general result on oscillatory integrals, whose proof we give in  Appendix \ref{ap:bound_highERes_pf}:

\begin{lemma}\label{lemma:vandercorput}
Let $\psi(k)$ be a smooth function supported in $ [2^j,2^{j+2}]\cup[-2^{j+2},-2^{j}]$ for $j\geq 1$, and for $j=0$ let $\psi(k)$ be supported in a neighborhood of the origin. Then, there is a constant $C>0$, which is independent of $\psi$, such that for all $r\in\R$ and all $t\in\R$:
\begin{align*}
    &\left\vert \int_{\R} e^{-it\sqrt{1+k^2}+ ikr}\psi(k)dk\right\vert \\
    &\leq C\min\left(\|\psi\|_{L^1},\vert t\vert^{-\frac12}2^{\frac{3}{2}j}\ \|\partial_k\psi\|_{L^1}, \vert t\vert^{-\frac32}2^{\frac32 j}\ \Big\|(\partial_{kk}+ir\partial_k)\left(\frac{\sqrt{1+k^2}}{k}\psi\right) \Big\|_{L^1}\right) \, .
\end{align*}
\end{lemma}

In order to bound \eqref{eq:Ij_def} via \eqref{eq:gen-form}, we apply Lemma \ref{lemma:vandercorput}, for each $j\ge0$ to
\begin{align}\label{eq:psij_highE}
   \psi= \psi_j(k) = \chi_j(k)L(k)\frac{k}{\sqrt{1+k^2}}\langle k\rangle ^{-3/2-\eps}\, .
\end{align}
Substituting \eqref{eq:psij_highE} into the  third norm expression in the upper bound of Lemma \ref{lemma:vandercorput}, we have that there is a constant, $C_\varepsilon$,  which is independent of $k$,  such that
\begin{align*}
    \Big\|\left(\partial_{kk}\pm  i( x\pm y)\partial_k \right)\left(\frac{\sqrt{1+k^2}}{k}\psi_j\right)\Big\|_{L^1} &= \left\|\left(\partial_{kk}\pm i(x\pm y)\partial_k\right)\left(\langle  k \rangle ^{-3/2-\eps} L(k)\chi_j\right)\right\|_{L^1} \\
    &\leq C_{\varepsilon} \left[\mathcal{I}_j^1 +\mathcal{I}_j ^2 + \mathcal{I}_j ^3 \right] \, ,
    \end{align*}
    where the $\mathcal{I}_j(x,y)$ are given by 
    \begin{subequations}\label{eq:Ij123}
    \begin{align}
    \mathcal{I}_j ^1 &\equiv \left\| \langle  k \rangle ^{-3/2-\eps}  \left[L''(k)\chi_j(k) +2L'(k)\chi_j(k) ' + L(k)\chi_j ''(k) \pm (x\pm y)\left(L'(k) \chi_j(k) + L (k)\chi_j '(k) \right)\right] \right\|_{L^1_k} \, , \\
    \mathcal{I}_j^2 &\equiv  \left\| k \langle k\rangle^{-7/2-\varepsilon} \left[2L'(k) \chi_j(k) + 2L(k) \chi_j '(k) \pm(x\pm y)L(k)\chi_j(k)  \right]\right\|_{L^1_k}  \, , \\
    \mathcal{I}_j^3 &\equiv  \left\| \left( \langle k\rangle^{-7/2+\varepsilon} + k^2 \langle k\rangle^{-11/2-\varepsilon} \right) L(k)\chi_j   (k)\right\|_{L^1_k} \, . 
     \end{align}
     \end{subequations}

\begin{prop}\label{lem:bound_highERes}
    There is a constant $C>0$ such that $\mathcal{I}_j^1, \mathcal{I}_j^2, \mathcal{I}_j^3\leq C 2^{-(\frac32+\varepsilon) j} \langle x \rangle \langle y \rangle$.
\end{prop}
Proposition \ref{lem:bound_highERes} is proved in  Appendix \ref{ap:bound_highERes_pf} using Van der Corput's Lemma on oscillatory integrals. 
 We finally apply the last upper bound in Lemma \ref{lemma:vandercorput} with $\psi=\psi_j$ given by \eqref{eq:psij_highE}, to bound $I_j$ defined in \eqref{eq:Ij_def}.  We obtain
\begin{align*}
    I_j (x,y) &\lesssim \left( \vert t\vert^{-\frac32}2^{\frac32 j}\right)\cdot \left[\mathcal{I}_j^1 +\mathcal{I}_j ^2 + \mathcal{I}_j ^3 \right] \lesssim \left( \vert t\vert^{-\frac32}2^{\frac32 j}\right)\cdot \left( 2^{-(\frac32+\eps) j} \langle x\rangle \langle y\rangle \right) =2^{-\eps j}  \vert t\vert^{-\frac32} \langle x\rangle \langle y\rangle .
\end{align*}
Therefore, by \eqref{eq:Ij_def} we have for any $x, y\in\R$
\begin{align*}
\langle x\rangle^{-1} |A_t(x,y)| \langle y\rangle^{-1} \le \sum_{j=0}^\infty
 \langle x\rangle^{-1} I_j(x,y) \langle y\rangle^{-1} \lesssim \vert t\vert^{-\frac32}, 
\end{align*}
since $\varepsilon$ is strictly positive. This proves the kernel bound \eqref{eq:kbound2}
 and therewith the high energy time-decay bound
 \eqref{eq:hi-wt} of Theorem \ref{thm:high_energy}.
 As remarked above, the bound \eqref{eq:hi-no-wt} follows by a closely related, but simpler, argument, involving the second upper bound in Lemma \ref{lemma:vandercorput}.
 This completes the proof of Theorem \ref{thm:high_energy}.
\end{proof}

\subsection{Low Energy Estimates}\label{sec:low_energy}

Concerning $\alpha_l(x,t)$, the low energy part of $\alpha(x,t)$, we have the following time-decay estimate:
\begin{theorem}\label{thm:loweng}
Let $\alpha_l(x,t)$ be defined as in \eqref{eq:low}.
There is a constant $C>0$ such that for $\tau\in[0,2\pi]$ we have the following weighted $L^1\to L^{\infty}$ bound
    \begin{align}
    \|\langle x\rangle^{-2} \alpha_l(x,t)\|_{L^{\infty}(\R_x)} \leq \frac{C}{|t|^{1/2}}\frac{1}{1+\sin^2(\tau/2) t}\|\langle y\rangle^2\alpha_0(y)\|_{L^1(\R_y)}.
\label{eq:low-bd}\end{align}
\end{theorem}

% Recall that our strategy is to prove a dispersive decay estimate separately for the low and high energy parts, $\alpha_l$ and $\alpha_h$, respectively \eqref{eq:highlow}. We now turn to the low-energy part $\alpha_l (t,x)$. 

In contrast to the high-energy time-decay bound (Theorem \ref{thm:high_energy}), the low-energy bound (Theorem \ref{thm:loweng}) depends on $\tau$, and in particular for (threshold) resonant values of $\tau$, the decay rate is only $\lesssim t^{-1/2}$.

%The point of departure for the proof of Theorem \ref{thm:loweng} is the representation of  $\alpha_l(x,t)$ as an integral over $k$ near zero, which via the relation $\omega= \pm\sqrt{1+k^2}$ corresponds to energies in near the endpoint of the essential spectrum, $\omega=\pm1$. That $\omega_{\tau}$, the eigenvalue energy, approaches $\pm1$ as $\tau\to0 ({\rm mod}\ 2\pi)$ (see Theorem \ref{prop:specD}), is reflected in the decay bound \eqref{eq:low-bd}. 

 The proof of Theorem \ref{thm:loweng} proceeds in the following way:
\begin{enumerate}
    \item Proposition \ref{lem:t12_loweng} establishes a $t^{-1/2}$ upper bound for all $\tau$ values.
    \item Proposition \ref{lem:t32_loweng} establishes a $t^{-3/2}$ upper bound for all $\tau \neq 0, 2\pi$.
    \item Finally, we use these upper bounds together with the interpolation inequality \eqref{eq:abc-ineq} to get the upper bound in Theorem \ref{thm:loweng}.
\end{enumerate}

 To capture the subtleties involving $k$ near zero (equivalently - energies near the band edge), we must carefully account for cancellations in $\Resk_{+}(k,\tau)-\Resk_-(k,\tau)$. The key is the jump formula of Proposition \ref{prop:limAbs}, which we use to reexpress the low energy representation formula \eqref{eq:low}:
\begin{equation*}
\alpha_l(t,x) =\frac{1}{2\pi i}\int_{0}^{\infty}e^{-i\sqrt{1+k^2} t}\chi(k)\ \langle k\rangle^{-3/2-\eps}\frac{k}{\sqrt{1+k^2}} \ \left[\left(\Resk_{+}(k,\tau)-\Resk_-(k,\tau)\right)\alpha_0\right] (x)  \, dk 
\end{equation*}
 By Proposition \ref{prop:limAbs}, we have 
  \begin{equation} \left[\left(\Resk_{+}(k,\tau)-\Resk_-(k,\tau)\right)\alpha_0\right](x) =
\frac{\vert T(k,\tau)\vert^2}{2i e^{-i\tau}k} \   \int\limits_{\R} f(x,y;k,\tau) \alpha_0(y) \, dy \, ,
\label{eq:R-jump}  \end{equation}
where
 \begin{equation} f(x,y;k,\tau)=\left(\xi_{+}(x;k,\tau)\xi_{+}(y;-k,\tau)^\top\sigma_1+e^{-i\tau}\xi_{-}(x;k,\tau)\xi_{-}(y;-k,\tau)^\top\sigma_1\right).
 \label{eq:f-def}\end{equation}
  Substitution of \eqref{eq:R-jump} into the definition of $\alpha_l(t,x)$
  (and cancelling factors of $k$),
  we obtain
  \begin{align}\label{eq:alphal_TF}
    \alpha_l(t,x) & = \frac{1}{-4\pi e^{-i\tau}}\int_0^{\infty} e^{-i\sqrt{1+k^2}t}\chi(k)\frac{\vert T(k,\tau)\vert^2}{\sqrt{1+k^2}}F(x;k,\tau) dk,
\end{align}
where 
\begin{align} \label{eq:Fdef}
    F(x;k,\tau) &= \langle k\rangle^{-3/2-\eps} \int_{\R} f(x,y;k,\tau)\ \alpha_0(y)dy\ .
\end{align}
We derive the desired low energy bound in two steps: first, in Proposition~\ref{lem:t12_loweng}, 
 we prove a $|t|^{-\frac12}$ decay-estimate for $\alpha_l(t,x)$ which holds uniformly in $\tau$.
Second, in Proposition~\ref{lem:t32_loweng}, we prove a $\tau$-dependent $|t|^{-\frac32}$ decay-estimate, as in Theorem \ref{maintheorem}. 

\begin{remark}
    In order to prove the aforementioned bounds, we must split the integral representation of $\alpha_{l}(x,t)$ further into two subdomains: $k$ values in a neighborhood of zero, and $k$ values which are small but bounded away from zero. This is necessary as there are no uniform in $\tau$ bounds near $k=0$ for $\partial_k\vert T(k,\tau)\vert^2$; see 
    \eqref{eq:dtk}. Therefore, one cannot apply a stationary phase argument for small values of $k$. For those values of $k$, a cruder estimate is used, and therefore this interval is chosen to be vanishingly small as $t\to \infty$.
\end{remark}

\begin{prop}[$\tau$-independent estimate]\label{lem:t12_loweng}
    Let $\alpha_l (t,x)$ be defined as in \eqref{eq:low} or equivalently \eqref{eq:alphal_TF}. There exists a constant $C_l>0$ such that for every $\tau \in [0,2\pi]$ and every $t>0$
    \begin{align}\label{eq:low-1}
        \|\langle x\rangle^{-1}\alpha_{l}(x,t)\|_{L^{\infty}(\R_x)} \leq \frac{C_l}{ | t|^{1/2}}\|\langle y\rangle \alpha_0(y)\|_{L^1(\R_y)}
    \end{align}
\end{prop}
\begin{proof}
It is useful to decompose $\alpha_l(x,t)$ further into its spectral contributions from: $0\le k\lesssim |t|^{-\frac12}$  and $|t|^{-\frac12}\lesssim k <\infty$ as follows
\begin{align}\label{eq:alphal_12}
\begin{split}
    \alpha_l(x,t) &=  \frac{1}{-4\pi e^{-i\tau}}\int_0^{\infty} e^{-i\sqrt{1+k^2}t}\chi(k)\chi(k\sqrt{t})\frac{\vert T(k,\tau)\vert^2}{\sqrt{1+k^2}}F(x;k,\tau) dk \\
    &+ \frac{1}{-4\pi e^{-i\tau}}\int_0^{\infty} e^{-i\sqrt{1+k^2}t}\chi(k)(1-\chi(k\sqrt{t}))\frac{\vert T(k,\tau)\vert^2}{\sqrt{1+k^2}}F(x;k,\tau) dk \\
    &:= \alo(x,t) + \alt(x,t).
\end{split}
\end{align}
To bound the first term, we note that since $|T(k,\tau)|\leq 1$ (Lemma \ref{lem:TransmissionBounds}), and since the integral is over an interval of length $\sim t^{-1/2}$, we get
\begin{align}\label{eq:alphal1_bd}
    \vert \alo(x,t)\vert &\leq \frac{1}{4\pi}\int_0^{2k_0/\sqrt{t}}\frac{\vert T(k,\tau)\vert^2}{\sqrt{1+k^2}}\vert F(x;k,\tau)\vert dk
    \leq \frac{k_0}{2\pi\sqrt{t}}\sup_{k} \vert F(x;k,\tau)\vert.
\end{align}

To bound the second term, we rewrite
\begin{align*}
    \vert \alt(x,t)\vert&= \left\vert \frac{1}{4\pi e^{-i\tau}}\int_0^{\infty} \frac{\sqrt{1+k^2}}{ik}\partial_k\left(e^{-i\sqrt{1+k^2}t}\right)\chi(k)(1-\chi(k\sqrt{t}))\frac{\vert T(k,\tau)\vert^2}{\sqrt{1+k^2}}F(x;k,\tau) dk \right\vert \, .
    \end{align*}
    Since the support of $\chi(k)(1-\chi(k\sqrt{t}))$ is bounded away from $k=0$ and $k=+\infty$, integration by parts and the triangle inequality yields
    \begin{align*}
        \vert \alt(x,t)\vert &= \left|\frac{1}{4\pi t}\int_0^{\infty} e^{-i\sqrt{1+k^2}t}\partial_k\left[\chi(k)(1-\chi(k\sqrt{t}))\frac{\vert T(k,\tau)\vert^2}{k}F(x;k,\tau)\right] dk \,  \right| \\
        &\leq \frac{1}{4\pi t}\int_0^{\infty}\left\vert e^{-i\sqrt{1+k^2}t}\partial_k\left[\chi(k)(1-\chi(k\sqrt{t}))\frac{\vert T(k,\tau)\vert^2}{k}F(x;k,\tau)\right]\right\vert dk \, , \\
        &\leq  \frac{1}{4\pi t}\int_{k_0/\sqrt{t}}^{2k_0}\left\vert\frac{\partial_k\left(\vert T(k,\tau)\vert^2F(x;k,\tau)\right)}{k}\right\vert +\left\vert \partial_k\left(\frac{\chi(k)(1-\chi(k\sqrt{t}))}{k}\right)\vert T(k,\tau)\vert^2F(x;k,\tau)\right\vert dk\\
    &:\,= {\rm (I)} + {\rm (II)}. \numberthis \label{eq:alphal2_split}
     \end{align*}
For (II) we estimate as follows, since the transmission coefficient satisfies $|T(k,\tau)| \leq 1$,
\begin{align*}
    {\rm (II)} &= \frac{1}{4\pi t}\int_{k_0/\sqrt{t}}^{2k_0}\left\vert \partial_k\left(\frac{\chi(k)(1-\chi(k\sqrt{t}))}{k}\right)\vert T(k,\tau)\vert^2F(x;k,\tau)\right\vert dk\\
    &=  \frac{1}{4\pi t}\sup\limits_{k\in [0,2k_0]}\vert F(x;k,\tau)\vert \int_{k_0/\sqrt{t}}^{2k_0}\left\vert \partial_k\left(\frac{\chi(k)(1-\chi(k\sqrt{t}))}{k}\right)\right\vert dk\\
    &= \frac{1}{4\pi t}\sup\limits_{k\in [0,2k_0]}\vert F(x;k,\tau)\vert \int_{k_0/\sqrt{t}}^{2k_0}\left\vert \left(\frac{\chi'(k)(1-\chi(k\sqrt{t}))+\sqrt{t}\chi(k)\chi'(k\sqrt{t})}{k}-\frac{\chi(k)(1-\chi(k\sqrt{t}))}{k^2}\right)\right\vert dk\\
    &\leq \frac{1}{4\pi t}\sup\limits_{k\in [0,2k_0]}\vert F(x;k,\tau)\vert \int_{k_0/\sqrt{t}}^{2k_0}\frac{\left\vert\chi'(k)\right\vert}{k}+\frac{\left\vert\sqrt{t}\chi'(k\sqrt{t})\right\vert}{k}+\frac{1}{k^2} dk \numberthis \label{eq:II_prebound} \, .
\end{align*}
To bound \eqref{eq:II_prebound}, we use that $\chi'(k)=0$ for $k\in [0,k_0)$ and $k\in~(2k_0, \infty)$ to get
\begin{align*}
\int_{k_0/\sqrt{t}}^{2k_0}\frac{\left\vert\chi'(k)\right\vert}{k}dk &\leq \sup\vert \chi'\vert\int_{k_0}^{2k_0}\frac{1}{k}dk =  \sup\vert \chi'\vert \cdot {\rm ln}(2)\ , \\
    \int_{k_0/\sqrt{t}}^{2k_0}\frac{\left\vert\sqrt{t}\chi'(k\sqrt{t})\right\vert}{k}dk &= \sqrt{t}\int_{k_0}^{2k_0\sqrt{t}}\frac{\vert \chi'(z)\vert}{z}dz\leq \sup\vert\chi'\vert\cdot\ln(2)\sqrt{t}\\
    \int_{k_0/\sqrt{t}}^{2k_0}\frac{1}{k^2} dk & = (\sqrt{t}-1/2)k_0^{-1}\leq \sqrt{t}k_0^{-1}
\end{align*}
Substituting these bounds into \eqref{eq:II_prebound}, we obtain
\begin{equation}\label{eq:II_loweng_bd}
({\rm II}) \leq  \frac{C'(1+k_0^{-1})}{\sqrt{t}}\sup\limits_{k\in [0,2k_0]}\vert F(x;k,\tau)\vert \,  .
\end{equation}
We now  bound the expression $({\rm I})$ in \eqref{eq:alphal2_split}. By applying the upper bounds on the transmission coefficient from Proposition \ref{lem:TransmissionBounds}, we get 
\begin{align*}
    ({\rm I}) &\leq \frac{1}{4\pi t}\int_{k_0/\sqrt{t}}^{2k_0}\vert \partial_k F(x;k,\tau)\vert \frac{\vert T(k,\tau)\vert^2}{k} + \vert\partial_k\vert T(k,\tau)\vert^2\vert \frac{\vert F(x;k,\tau)\vert}{k}dk\\
    &\leq \frac{1}{4\pi t}\sup\limits_{k\in [0,2k_0]}\vert \partial_k F(x;k,\tau)\vert\int_{k_0/\sqrt{t}}^{2k_0} \frac{1}{k}dk + \frac{1}{4\pi t}\sup\limits_{k\in [0,2k_0]}\vert F(x;k,\tau)\vert\int_{k_0/\sqrt{t}}^{2k_0}\frac{1}{k^2}dk\\
     &\leq \frac{1}{4\pi t}\sup\limits_{k\in [0,2k_0]}\vert \partial_k F(x;k,\tau)\vert(1+\ln(t)) + \frac{1}{4\pi t}\sup\limits_{k\in [0,2k_0]}\vert F(x;k,\tau)\vert(1+\sqrt{t})\\
     &\leq \frac{C''}{\sqrt{t}}\left(\sup\limits_{k\in [0,2k_0]}\vert \partial_k F(x;k,\tau)\vert+\sup\limits_{k\in [0,2k_0]}\vert F(x;k,\tau)\vert\right) \numberthis \label{eq:I_upperbound}
\end{align*}
By the upper bound on $|\alpha_l^{(1)}(x,t)| $, \eqref{eq:alphal1_bd}, and the upper bound on $|\alpha_l^{(2)}(x,t)| \le {\rm (I)} + {\rm (II)}$, there is a constant $C>0$ (independent of $\alpha_{l}$, $t$, and $\tau$), such that
\[    |\alpha_{l}(x,t)| \leq \frac{C}{\sqrt{t}} \left(  \sup\limits_{k\in [0,2k_0]} | F(x;k, \tau) | +\sup\limits_{k\in [0,2k_0]}|\partial_k  F(x;k, \tau) | \right) \, . \]
Hence, to complete the proof of Proposition \ref{lem:t12_loweng}, we need to bound $F$ and $\partial_k F$ in terms of the $L^1$ norm of $\alpha_0$. These bounds are displayed in
Proposition \ref{lem:Fbounds} of \Cref{ap:Fbounds}. 
\end{proof}

Next we prove a more subtle $\tau$-dependent decay estimate, which requires stronger localization of $\alpha_0$. We shall make crucial use of the small $|k|$ behavior of the transmission coefficient (see \eqref{eq:tk}):
\begin{equation}
    \vert T(k,\tau)\vert^2 = \frac{k^2}{k^2+\sin^2(\tau/2)}.
\label{eq:tk-1}\end{equation} 

\begin{prop}[$\tau$-dependent estimate]\label{lem:t32_loweng}
    Let  $\alpha_l(x,t)$ denote low-energy part of $\alpha(x,t)$; see \eqref{eq:low} or equivalently \eqref{eq:alphal_TF}. There is a $\tau$-independent constant $D>0$ such that for $t>0$ and $\tau \in (0,2\pi)$ (no threshold resonance) we have:
    \begin{align*}
            \|\langle x\rangle^{-2} \alpha_l(x,t)\|_{L^{\infty}(\R_x)} \leq \frac{D}{\sin^2(\tau/2)\ |t|^{\frac32}}\ \|\langle y\rangle^2\alpha_0(y)\|_{L^1(\R_y)}\ ,
    \end{align*}
\end{prop}
\begin{proof}

 Assume  $\tau \in (0,2\pi)$. We study the low-energy solution $\alpha_l(t,x)$ as displayed in \eqref{eq:alphal_TF}. Using integration by parts we have
\begin{align*}
    \alpha_l(t,x) &=  \frac{1}{-4\pi e^{-i\tau}}\int_0^{\infty} e^{-i\sqrt{1+k^2}t}\chi(k)\frac{\vert T(k,\tau)\vert^2}{\sqrt{1+k^2}}F(x;k,\tau) dk\\
    &=\frac{1}{-4\pi e^{-i\tau}}\int_0^{\infty} \frac{\sqrt{1+k^2}}{-ikt}\partial_k\left(e^{-i\sqrt{1+k^2}t}\right)\chi(k)\frac{\vert T(k,\tau)\vert^2}{\sqrt{1+k^2}}F(x;k,\tau) dk\\
    &=\frac{1}{4\pi it e^{-i\tau}}\int_0^{\infty} \partial_k\left(e^{-i\sqrt{1+k^2}t}\right)\chi(k)\frac{\vert T(k,\tau)\vert^2}{k}F(x;k,\tau) dk\\
    &=-\frac{1}{4\pi i t e^{-i\tau}}\int_0^{\infty} e^{-i\sqrt{1+k^2}t}\partial_k\left(\chi(k)\frac{\vert T(k,\tau)\vert^2}{k}F(x;k,\tau)\right) dk  \\
    &+ e^{-i\sqrt{1+k^2}t}\chi(k)\frac{\vert T(k,\tau)\vert^2}{k}F(x;k,\tau)\bigg|_{k=0}^{k=\infty}
    \end{align*}
Now note that the the boundary terms vanish; At $k=\infty$, this is due to the vanishing of $\chi(k)$ as $k\to+\infty$. At $k=0$, the boundary term vanishes since $\vert T(k,\tau)\vert^2=\mathcal{O}(k^2)$ as $|k|\to0$ for $0<\tau<2\pi$, by  \eqref{eq:tk-1}.
 Since the boundary terms vanish, it follows that 
 \begin{align*}
   \alpha_l(t,x) &= \frac{1}{-4\pi i t e^{-i\tau}}\int_0^{\infty} e^{-i\sqrt{1+k^2}t}G(x;k,\tau) dk \, , \numberthis \label{eq:alphal_dkTF}
\end{align*}
where
\begin{align}\label{eq:Gdef}
G(x;k,\tau) \equiv \partial_k\left[\chi(k)\frac{\vert T(k,\tau)\vert^2}{k}F(x;k,\tau)\right]\, . 
\end{align}

\begin{remark}
Here we see why Proposition \ref{lem:t32_loweng} holds only for $\tau\in (0,2\pi)$: for the boundary terms to vanish at $k=0$, we need that $T(0,\tau)=0$. This is {\em not true} when $\tau = 0,2\pi$, since the transmission coefficient there is always $1$.
\end{remark}

The new representation of $\alpha_l$ given in \eqref{eq:alphal_dkTF} can be decomposed, in analogy with \eqref{eq:alphal_12}, into $k$ near zero ($k< t^{-1/2}k_0$) and $k$ bounded away from zero ($k\in (t^{-1/2}k_0, 2k_0)$).
For every $\tau \in (0,2\pi)$,
\begin{align}\label{eq:alphal_i_ii}
\begin{split}
    \alpha_l(t,x) &= \frac{1}{-4\pi i t e^{-i\tau}}\int_0^{\infty} e^{-i\sqrt{1+k^2}t}\chi(\sqrt{t}k) G(x;k,\tau) dk\\
     &+ \frac{1}{-4\pi i t e^{-i\tau}}\int_0^{\infty} e^{-i\sqrt{1+k^2}t}(1-\chi(\sqrt{t}k)) G(x;k,\tau) dk\\
     &=\ali(x,t) + \alii(x,t) \, .
     \end{split}
\end{align}
We estimate each term in \eqref{eq:alphal_i_ii}. First,
\begin{align}\label{eq:alphal_i_bd}
    \vert \ali(t,x)\vert &\leq \frac{1}{4\pi t}\int_0^{2k_0/\sqrt{t}}\vert G(x;k,\tau)\vert dk\ 
    \leq\ \frac{k_0}{2\pi t^{3/2}}\sup_{k}\vert G(x;k,\tau)\vert \, .
\end{align}
To estimate $\alii$ we first use integration by parts,
\begin{align*}
    \alii(t,x) &=  \frac{1}{-4\pi i t e^{-i\tau}}\int_0^{\infty} e^{-i\sqrt{1+k^2}t}(1-\chi(\sqrt{t}k)) G(x;k,\tau) dk\\
    &= \frac{1}{-4\pi i t e^{-i\tau}}\int_0^{\infty} \frac{\sqrt{1+k^2}}{-ikt}\partial_k\left(e^{-i\sqrt{1+k^2}t}\right)(1-\chi(\sqrt{t}k)) G(x;k,\tau) dk\\
    &= \frac{1}{4\pi t^2 e^{-i\tau}}\int_0^{\infty} e^{-i\sqrt{1+k^2}t}\partial_k\left(\frac{\sqrt{1+k^2}}{k}(1-\chi(\sqrt{t}k)) G(x;k,\tau)\right) dk \, .
\end{align*}
It follows that
\begin{align*}
    \vert\alii(t,x)\vert &\leq \frac{1}{4\pi t^2}\int_{k_0/\sqrt{t}}^{2k_0}\left\vert \partial_k\left(\frac{\sqrt{1+k^2}}{k}(1-\chi(\sqrt{t}k))\right) G(x;k,\tau)\right\vert dk\\
    &+\frac{1}{4\pi t^2}\int_{k_0/\sqrt{t}}^{2k_0} \left\vert \left(\frac{\sqrt{1+k^2}}{k}(1-\chi(\sqrt{t}k)) \partial_kG(x;k,\tau)\right)\right\vert dk\\
    &\le {\rm Term}_1 + {\rm Term}_2 + {\rm Term}_3,
    \end{align*} 
    where
    \begin{align*}
   {\rm Term}_1 &\equiv \frac{1}{4\pi t^2}\sup_{k}\vert G(x;k,\tau)\vert\int_{k_0/\sqrt{t}}^{2k_0}\left\vert \frac{\sqrt{1+k^2}}{k}\sqrt{t}\vert\chi'(\sqrt{t}k))\vert \right\vert \, dk \, ,\\
      {\rm Term}_2 &\equiv \frac{1}{4\pi t^2}\sup_{k}\vert G(x;k,\tau)\vert\int_{k_0/\sqrt{t}}^{2k_0}\frac{1}{k^2\sqrt{1+k^2}} \, dk \, ,\quad {\rm and} \\
      {\rm Term}_3 &\equiv\ \frac{1}{4\pi t^2}\int_{k_0/\sqrt{t}}^{2k_0} \frac{\sqrt{1+k^2}}{k} \left\vert\partial_kG(x;k,\tau)\right\vert \, dk \, .
\end{align*}
Recall that $G(x;k,\tau)$ is defined in \eqref{eq:Gdef} in terms of $F(x;k,\tau)$, which depends on the initial data, $\alpha_0$; see \eqref{eq:Fdef}.

We now provide upper bounds for each of the three terms above. These upper bounds rely on
Proposition \ref{lem:Fbounds}, which provides bounds on $G$ and its derivatives. By  changing variables ($z:\,= k\sqrt{t}$), and using that $\chi'(z)$ is compactly supported and vanishes in a neighborhood of $z=0$, we obtain (with positive constants $D_1$, $D_2$, etc., all independent of $k$ and $\tau$):
\begin{align}\label{eq:t32bd_1}
\begin{split}
    {\rm Term}_1 &\leq \frac{D_1}{t^{3/2}}\frac{\langle x\rangle}{\sin^2(\tau/2)} \
    \|\langle y\rangle\alpha_0(y)\|_{L^1(\R_y)}\ \int_{k_0/\sqrt{t}}^{2k_0} \frac{\sqrt{1+k^2}}{k}\left\vert\chi'(\sqrt{t}k)) \right\vert \,  dk \\
    &= \frac{D_1}{t^{3/2}}\frac{\langle x\rangle}{\sin^2(\tau/2)}\ \|\langle y\rangle\alpha_0(y)\|_{L^1(\R_y)}\ \int_{k_0}^{2k_0\sqrt{t}} \frac{\sqrt{1+z^2/t}}{z}\left\vert\chi'(z)) \right\vert  \, dz\\
    &\leq \frac{D_2}{t^{3/2}}\frac{\langle x\rangle}{\sin^2(\tau/2)}\
    \|\langle y\rangle\alpha_0(y)\|_{L^1(\R_y)}\, .
    \end{split}
\end{align}

Similarly,
\begin{align}\label{eq:t32bd_2}
\begin{split}
     {\rm Term}_2  &\leq \frac{D_3}{t^2}\frac{\langle x\rangle}{\sin^2(\tau/2)}\  \|\langle y\rangle\alpha_0(y)\|_{L^1(\R_y)}\ \int_{k_0/\sqrt{t}}^{2k_0}\frac{1}{k^2\sqrt{1+k^2}} \, dk\\ 
    &\leq\ \frac{D_3}{t^2}\frac{\langle x\rangle}{\sin^2(\tau/2)}\  \|\langle y\rangle\alpha_0(y)\|_{L^1(\R_y)}\ \int_{k_0/\sqrt{t}}^{2k_0}\frac{1}{k^2}  \, dk\\
    &\leq \frac{D_4}{t^{3/2}}\frac{\langle x\rangle}{\sin^2(\tau/2)}\  \|\langle y\rangle\alpha_0(y)\|_{L^1(\R_y)} \, .
    \end{split}
\end{align}

 Finally, using the change of variables $k= (k_0/\sqrt t) l$,
\begin{align}\label{eq:t32bd_3}
\begin{split}
    {\rm Term}_3 &\leq \frac{D_5}{t^2}\langle x\rangle^2\ \|\langle y\rangle^2 \alpha_0(y)\|_{L^1(\R_y)}\ \int_{k_0/\sqrt{t}}^{2k_0} \frac{\sqrt{1+k^2}}{k}\frac{1}{k(k^2+\sin^2(\tau/2))} \,  dk\\
    &=\frac{D_5}{t^2}\langle x\rangle^2 \ \|\langle y\rangle^2 \alpha_0(y)\|_{L^1(\R_y)}\ \int_{k_0/\sqrt{t}}^{2k_0} \frac{\sqrt{1+k^2}}{k}\frac{1}{k(k^2+\sin^2(\tau/2))} \,  dk \\
    &=\frac{D_6}{t^{1/2}}\ \langle x\rangle^2\ \|\langle y\rangle^2 \alpha_0(y)\|_{L^1(\R_y)}\ \int_{1}^{2\sqrt{t}} \frac{1}{l^2}\frac{1}{(k_0^2 l^2+\sin^2(\tau/2)t)} \,  dl\\
    &\leq \frac{D_6}{t^{1/2}}\ \langle x\rangle^2\ \|\langle y\rangle^2 \alpha_0(y)\|_{L^1(\R_y)}\ 
    \frac{1}{k_0^2 +\sin^2(\tau/2)t}\int_{1}^{2\sqrt{t}} \frac{1}{l^2}  \, dl\\
     &\leq \frac{D_6}{\sin^2(\tau/2)\ t^{3/2}}\ 
      \langle x\rangle^2 \ \|\langle y\rangle^2 \alpha_0(y)\|_{L^1(\R_y)} \, .
     \end{split}
\end{align}

To conclude,  we combine the upper bound \eqref{eq:alphal_i_bd}  for $\ali$, and the three upper bounds which \eqref{eq:t32bd_1}, \eqref{eq:t32bd_2}, and \eqref{eq:t32bd_3} together bound $\alii$, to get by \eqref{eq:alphal_i_ii} and the triangle inequality
\begin{align*}%\label{eq:secondestimate}
|\alpha_l(t,x)|\ \le\     \frac{D_7}{\sin^2(\tau/2)\ t^{3/2}}\ 
      \langle x\rangle^2 \ \|\langle y\rangle \alpha_0(y)\|_{L^2(\R_y)}.
\end{align*}

\end{proof}

Finally, to complete the proof of Theorem \ref{thm:loweng}, we interpolate via inequality \eqref{eq:abc-ineq},
using the bounds of  Propositions  \ref{lem:t12_loweng} and \ref{lem:t32_loweng}. 
This yields \begin{align*}
    \|\langle x\rangle^{-2} \alpha_l(t,x)\|_{L^{\infty}(\R_x)} \leq \frac{D}{t^{1/2}}\frac{1}{1+\sin^2(\tau/2) t}\|\langle y\rangle\alpha_0(y)\|_{L^1(\R_y)}\ .
\end{align*}
which implies the bound of Theorem \ref{thm:loweng}.

\subsection{Remarks on the proof of \Cref{notmaintheorem}} \label{sec:notmaintheorem}

By the interpolation argument at the beginning of \Cref{sec:mainproof}, it suffices to prove both results in \Cref{notmaintheorem} in a weaker version, where $\langle t\rangle$ replaced by $\vert t \vert$. Moreover, these bounds are established with $\tau$-independent constants for high energies in \Cref{thm:high_energy}. Thus, it only remains to prove the desired inequalities with $\tau$-dependent coefficients for the low energy part of the solution $\alpha_l(x,t)$, displayed  in \eqref{eq:alphal_TF}.

The desired inequalities follow from an application of \Cref{lemma:vandercorput} similar to that in \Cref{thm:high_energy}. The middle bound provided by \Cref{lemma:vandercorput} provides the necessary inequality for \eqref{eq:noweights} while the third bound in \Cref{lemma:vandercorput} provides the inequality for \eqref{eq:betterWeight}.

%\subsection*{Data availability and conflict of interest} The authors have no competing interests to declare that are relevant to the content of this article. Data sharing not applicable to this article as no datasets were generated or analysed during the current study.
\appendix

\section{Proof of Proposition \ref{prop:xi_pm}}\label{app:pf_prop_xipm}
\begin{proof}
The expressions \eqref{eq:xi+} for $\xi_+$
 and \eqref{eq:xi-} for $\xi_-$
 are solutions with the desired asymptotic behavior at infinity if and only if each of these expressions is continuous at $x=0$.
 This imposes linear inhomogeneous systems of algebraic equations  for $(A,B)$ and for $(C,D)$. The solutions are displayed \eqref{eq:b}. This proves Part 1 of 
 Proposition \ref{prop:xi_pm}.
%

%We define solutions $\xi_{\pm}(x;k,\tau)$ to
%\begin{align}
%    \partial_x \xi_{\pm} = M(k,\tau)\xi_{\pm}\ ,
%\end{align}
%such that
%\begin{align}
%    \lim_{x\to+\infty}\xi_{+}(x;k,\tau) &\sim \begin{pmatrix}
%    e^{-i\tau}\\ \sqrt{1+k^2}+k
%    \end{pmatrix}e^{+ikx}\ , \\
%      \lim_{x\to-\infty}\xi_{-}(x;k,\tau) &\sim \begin{pmatrix}
%      1 \\ \sqrt{1+k^2}-k
%      \end{pmatrix}e^{-ikx}.
%\end{align}
%Here the matrix $M(k,\tau)$ is defined
%\begin{align} 
%M(k,\tau) = \begin{cases}M_-(x;k,\tau)& x < 0\\ M_+(k,\tau) & x > 0
%    \end{cases}.
%\end{align}

Next we turn our attention to the relations \eqref{eq:Sxi_symmetries} between $\xi_+$ and $\xi_-$. Note that \eqref{eq:Sxi_symmetries2} is equivalent to  \eqref{eq:Sxi_symmetries1}.  Indeed, 
\eqref{eq:Sxi_symmetries1} implies
 \[ \overline{S(\tau)\xi_+(x;\omega,\tau)} = 
 \xi_-(-x,\overline{\omega},\tau).
 \]
 Making the replacements $x\mapsto -x$ and $\omega\to\overline{\omega}$ gives \eqref{eq:Sxi_symmetries2}. Reversing the steps recovers \eqref{eq:Sxi_symmetries1} from \eqref{eq:Sxi_symmetries2}.
 
The last item in Proposition \ref{prop:xi_pm} is the identity
 \eqref{eq:Sxi_symmetries2}, rewritten as:
 \begin{align}\label{eq:Sxi} \overline{S(\tau)\xi_{+}(-x;\overline{\omega},\tau)} &= \xi_{-}(x;\omega,\tau). \end{align}
 We will first  show  that the expressions on either side of the equality 
 \eqref{eq:Sxi} satisfy the same differential equation and  asymptotic condition at $+\infty$. Once we do that, uniqueness then implies
the equality of these expressions. 
First, to verify the equality of asymptotic behaviors at infinity, recall by \eqref{xi-bc} that 
\begin{equation}\label{c1} 
e^{-\lo x} \xi_{-}(x;\omega,\tau) \longrightarrow
v_-^{(+)}=\begin{pmatrix}
    1\\ \omega -i\lo
    \end{pmatrix}\ \textrm{as $x\to-\infty$.}  \end{equation}
     Concerning the left hand side of \eqref{eq:Sxi}, note that by \eqref{xi+bc}
     \[S(\tau)e^{\lo(-x)}\xi_+(-x,\omega,\tau) \to \begin{pmatrix}
    1\\ \omega+i\lo
    \end{pmatrix}\ \textrm{as $x\to-\infty$.}  \]
    Hence,
     \[\overline{S(\tau)e^{\lo(-x)}\xi_+(-x,\omega,\tau)} \to \begin{pmatrix}
    1\\ \bar\omega-i\lambda(\bar\omega)
    \end{pmatrix}\ \textrm{as $x\to-\infty$}  \]
    or 
     \[e^{-\lambda(\bar\omega)x}\overline{S(\tau)\xi_+(-x,\omega,\tau)} \to \begin{pmatrix}
    1\\ \bar\omega-i\lambda(\bar\omega)
    \end{pmatrix}\ \textrm{as $x\to-\infty$}  \]
    or, replacing $\omega$ by $\bar\omega$, we obtain
     \begin{equation}\label{c2}
     e^{-\lambda(\omega)x}\overline{S(\tau)\xi_+(-x,\bar\omega,\tau)} \to \begin{pmatrix}
    1\\ \omega-i\lambda(\omega)
    \end{pmatrix}\ \textrm{as $x\to-\infty$}  \end{equation}
    Therefore, by \eqref{c1} and \eqref{c2}, 
    $\overline{S(\tau)\xi_{+}(-x;\overline{\omega},\tau)}$ and 
    $\xi_{-}(x;\omega,\tau)$ satisfy the same asymptotic condition as $x\to-\infty$.

To complete the proof we now verify that both expressions in the equality \eqref{eq:Sxi}, $\overline{S(\tau)\xi_{+}(-x;\overline{\omega},\tau)}$ and 
    $\xi_{-}(x;\omega,\tau)$,  satisfy the same differential equation. We begin by first noting the following identities:
    \begin{equation}\label{eq:3id} S\sigma_3=\sigma_3 S,\quad S\sigma_1=\overline{\sigma_\star}S,\quad 
 S\sigma_\star=\sigma_1S,\end{equation}
 where we have used the abbreviated notation
  $S=S(\tau)$, $\sigma_\star=\sigma_\star(\tau)$, and $\D=\D(\tau)$. Using \eqref{eq:3id} we have
  \begin{align*}
    S\D\xi_{+}(x;\omega,\tau)=\left[i\sigma_3\partial_x + \overline{\sigma_\star}\mathbbm{1}_{x<0} + \sigma_1\mathbbm{1}_{x>0} \right]S\xi_{+}(x;\omega,\tau) = \omega S\xi_{+}(x;\omega,\tau)
  \end{align*}
  Changing variables: $x=-y$ gives
  \begin{align*}
    \left[-i\sigma_3\partial_y  + \sigma_1\mathbbm{1}_{y<0} + \overline{\sigma_\star}\mathbbm{1}_{y>0} \right]S\xi_{+}(-y;\omega,\tau) = \omega\ S\xi_{+}(-y;\omega,\tau),\ 
  \end{align*}
 taking the complex conjugate gives
  \begin{align*}
    \left[i\sigma_3\partial_y + \sigma_1\mathbbm{1}_{y<0} + \sigma_\star\mathbbm{1}_{y>0} 
    \right]\overline{S\xi_{+}(-y;\omega,\tau)} = \overline{\omega}\ \overline{S\xi_{+}(-y;\omega,\tau)}
  \end{align*}
  and finally making the replacement $\omega\to -\overline{\omega}$ gives
  \begin{align*}
 \D \overline{S\xi_{+}(-y;\overline{\omega},\tau)}  =\left[i\sigma_3\partial_y + \sigma_1\mathbbm{1}_{y<0} +\sigma_\star\mathbbm{1}_{y>0} \right]\overline{S\xi_{+}(-y; \overline{\omega},\tau)} = \omega\ \overline{S\xi_{+}(-y;\overline{\omega},\tau)}.
  \end{align*}
  Since $\xi_-(x;\omega,\tau)$ satisfies the same equation and boundary condition at $-\infty$, the proof of the symmetry properties \eqref{eq:Sxi_symmetries} is now complete.
  \end{proof}

\section{Spectral symmetry of the operator $\D(\tau)$} \label{app:spectralsymmetry}
We will show that there is a symmetry of the eigenvalue equation 
\begin{align}
    \D(\tau)\alpha = \omega\alpha\ .
\end{align}
Let $\beta(x)=\sigma_3 S(\tau)\alpha(-x)$. Then we claim that $\beta$ satisfies
\begin{align*}
    \D(2\pi-\tau)\alpha = -\omega\beta\ .
\end{align*}
To see this, replace $x\mapsto -x$ in the original equation, multiply by $\sigma_3 S(\tau)$, and use the anticommutation relations satisfied by Pauli matrices to observe
\begin{align}
    \omega\sigma_3 S(\tau)\alpha(-x) &= \sigma_3 S(\tau)(-i\sigma_3\partial_x  + \sigma_1 \mathbbm{1}_{(-\infty,0)}(-x)  + \sigma_\star(\tau)\mathbbm{1}_{[0,\infty)}(-x))\alpha(-x)\\
     &=\sigma_3 S(\tau)(-i\sigma_3\partial_x  + \sigma_1 \mathbbm{1}_{(0,\infty)}(x)  + \sigma_\star(\tau)\mathbbm{1}_{(-\infty,0]}(x))\alpha(-x)\\
     &=\sigma_3 (-i\sigma_3\partial_x  + \sigma_{\star}(2\pi-\tau) \mathbbm{1}_{(0,\infty)}(x)  + \sigma_1\mathbbm{1}_{(-\infty,0]}(x))S(\tau)\alpha(-x)\\
     &=(-i\sigma_3\partial_x  - \sigma_{\star}(2\pi-\tau) \mathbbm{1}_{(0,\infty)}(x)  - \sigma_1\mathbbm{1}_{(-\infty,0]}(x))\sigma_3 S(\tau)\alpha(-x) \\
     &=-\D(2\pi-\tau)\sigma_3 S(\tau)\alpha(-x)\ .
\end{align}
Making the definition $\beta(x)= \sigma_3 S(\tau)\alpha(-x)$ and multiplying by minus $1$, we obtain the desired equation.

%%%%%%%%%%%%%%%%%%%%%%%%%%%%%%%%%%%%%%%%%%%%%%%%%%%%%%%%%%%%%%%%%%%%%%%%%%
\section{Limiting Resolvent matrix elements}\label{ap:resolvent_explicit}
For $k\in\R$, the limiting resolvents $\Resk_{\pm}(k,\tau)$ are given by
\begin{align}
    \Resk_{\pm}(k,\tau) &=\begin{cases} \frac{-i}{\varphi(\pm k,\tau)}\xi_{-}(x; \pm k,\tau)\xi_{+}(y; \pm k,\tau)^\top\sigma_1 & x < y\\
    \frac{-i}{\varphi(\pm k,\tau)}\xi_{+}(x; \pm k,\tau)\xi_{-}(y; \pm k,\tau)^\top\sigma_1 & x > y
    \end{cases}
\end{align}
Using the explicit formulas for the Jost solutions \eqref{eq:xi_pm}, we can write the Resolvent entries in closed form. Let $\Resk(k,\tau)$ be defined such that
\begin{align}
    \Resk_{\pm}(k,\tau) = \Resk(\pm k,\tau) . 
\end{align}
Then the entries of the kernel of $\Resk(k,\tau)$ are, (where the $\tau$-dependent coefficients are determined by \eqref{eq:b} as $\FF(k)=A(\sqrt{1+k^2})$, and similarly for $\FF, \GG, \HH$):
\begin{align}\label{eq:explicitresolvent}
    (\Resk(k,\tau)(x,y))_{1,1} &= \begin{cases}
    \frac{i}{\varphi(k,\tau)}e^{-i\tau}\left[\GG (\sqrt{1+k^2}+k)e^{ik(x+y)}+\HH (\sqrt{1+k^2}-k)e^{ik(x-y)}\right]& x > y > 0\\
    \frac{i}{\varphi(k,\tau)}e^{-i\tau}(\sqrt{1+k^2}-k)e^{ik(x-y)}& x > 0 > y\\
    \frac{i}{\varphi(k,\tau)}\left[\EE (\sqrt{1+k^2}-k)e^{ik(x-y)}+\FF (\sqrt{1+k^2}-k)e^{-ik(x+y)}\right]& 0 > x > y\\
    \frac{i}{\varphi(k,\tau)}\left[\GG(\sqrt{1+k^2}+k)e^{-i\tau}e^{ik(x+y)}+\HH (\sqrt{1+k^2}+k)e^{-i\tau}e^{-ik(x-y)}\right]& y > x > 0\\
    \frac{i}{\varphi(k,\tau)}(\sqrt{1+k^2}+k)e^{-ik(x-y)}& y > 0 > x\\
    \frac{i}{\varphi(k,\tau)}\left[\EE (\sqrt{1+k^2}+k)e^{-ik(x-y)}+\FF (\sqrt{1+k^2}-k)e^{-ik(x+y)}\right] & 0 > y > x
    \end{cases}\ ,
\end{align}

\begin{align}
    (\Resk(k,\tau)(x,y))_{1,2} &= \begin{cases}
    \frac{i}{\varphi(k,\tau)}e^{-2i\tau}\left[\GG e^{ik(x+y)}+\HH e^{ik(x-y)}\right] & x > y > 0\\
    \frac{i}{\varphi(k,\tau)}e^{-i\tau}e^{ik(x-y)} & x > 0 > y\\
    \frac{i}{\varphi(k,\tau)}\EE e^{ik(x-y)}+\FF e^{-ik(x+y)} & 0 > x > y\\
    \frac{i}{\varphi(k,\tau)}e^{-2i\tau}\left[\GG e^{ik(x+y)}+\HH e^{-ik(x-y)}\right] & y > x > 0\\
    \frac{i}{\varphi(k,\tau)}e^{-i\tau}e^{-ik(x-y)} & y > 0 > x\\
    \frac{i}{\varphi(k,\tau)}e^{-2i\tau}\left[\GG e^{ik(x+y)}+\HH e^{ik(x-y)}\right] & 0 > y > x
    \end{cases}\ ,
\end{align}
\begin{align}
    (\Resk(k,\tau)(x,y))_{2,1} &= \begin{cases}
     \frac{i}{\varphi(k,\tau)}\GG(\sqrt{1+k^2}+k)^2e^{ik(x+y)}+\HH e^{ik(x-y)}& x > y > 0\\
     \frac{i}{\varphi(k,\tau)}e^{ik(x-y)}& x > 0 > y\\
     \frac{i}{\varphi(k,\tau)}\EE e^{ik(x-y)}+\FF(\sqrt{1+k^2}-k)^2e^{-ik(x+y)}& 0 > x > y\\
     \frac{i}{\varphi(k,\tau)}\GG(\sqrt{1+k^2}+k)^2e^{ik(x+y)}+\HH e^{-ik(x-y)}& y > x > 0\\
     \frac{i}{\varphi(k,\tau)}e^{-ik(x-y)}& y > 0 > x\\
     \frac{i}{\varphi(k,\tau)}\EE e^{-ik(x-y)}+\FF(\sqrt{1+k^2}-k)^2e^{-ik(x+y)}& 0 > y > x
    \end{cases}\ ,
\end{align}
\begin{align}
    (\Resk(k,\tau)(x,y))_{2,2} &= \begin{cases}
     \frac{i}{\varphi(k,\tau)}\GG e^{-i\tau}(\sqrt{1+k^2}+k)e^{ik(x+y)}+\HH e^{-i\tau}(\sqrt{1+k^2}+k)e^{ik(x-y)} & x > y > 0\\
     \frac{i}{\varphi(k,\tau)}(\sqrt{1+k^2}+k)e^{ik(x-y)}& x > 0 > y\\
     \frac{i}{\varphi(k,\tau)}\left[\EE(\sqrt{1+k^2}+k)e^{ik(x-y)}+\FF (\sqrt{1+k^2}-k)e^{-ik(x+y)}\right]& 0 > x > y\\
     \frac{i}{\varphi(k,\tau)}e^{-i\tau}\left[\GG(\sqrt{1+k^2}+k)e^{ik(x+y)}+\HH (\sqrt{1+k^2}-k)e^{-ik(x-y)}\right]& y > x > 0\\
     \frac{i}{\varphi(k,\tau)}e^{-i\tau}(\sqrt{1+k^2}-k)e^{-ik(x-y)}& y > 0 > x\\
     \frac{i}{\varphi(k,\tau)}\left[\EE (\sqrt{1+k^2}-k)e^{-ik(x-y)}+\FF (\sqrt{1+k^2}-k)e^{-ik(x+y)}\right]& 0 > y > x
    \end{cases} \, ,
\end{align}
where as in \eqref{eq:varphi_def}, substituting $\omega=\sqrt{1+k^2}$ for real values of $k$, we get
\begin{align}\label{eq:varphik_def}
    \varphi(k,\tau) :\,= k(e^{-i\tau}+1)  -\sqrt{1+k^2}(e^{-i\tau}-1)\ ,
\end{align}

%%%%%%%%%%%%%%%%%%%%%%%%%%%%%%%%%%%%%%%%%%%%%%%%%%%%%%%%%%%%%%%%%%%%%%%%%%
\section{Proofs of Lemmas from Section \ref{sec:high_energy}}\label{ap:bound_highERes_pf}
%%%%%%%%%%%%%%%%%%%%%%%%%%%%%%%%%%%%%%%%%%%%%%%%%%%%%%%%%%%%%%%%%%%%%%%%%%
\subsection{Proof of Lemma \ref{lemma:vandercorput}}
    The first term is a direct consequence of the triangle inequality $$\left|\int\limits_{\R} e^{ -it\sqrt{1+k^2}+ikr}\psi_j (k) \, dk \right| \leq \int\limits_{\R} \left| e^{ -it\sqrt{1+k^2}+ikr}\right| \cdot \left| \psi_j (k)\right| \, dk = \left\| \psi _j \right\|_{L^1} \, .  $$
    The next two upper bounds are the result of Van der Corput's Lemma for oscillatory integrals:
\begin{theorem}[Van der Corput Lemma \cite{stein1993harmonic}]
    Let $\phi:\R \to \R$ be a smooth and let $\psi \in C^{\infty}_c (\R )$. If $|\partial_k^2\phi(k)|\geq \lambda \geq 0$ for all $k\in {\rm supp}(\psi)$, then there exists a constant $C>0$ such that
    \begin{equation}\label{eq:vdc}
        \left| \int\limits_{\R} 
        e^{i\phi (k)}\psi(k) \, dk \right| \leq C \lambda^{-\frac12} \left\| \partial_k \psi(k) \right\|_{L^1{(\R_k)}} \, .
    \end{equation}
\end{theorem}
By \eqref{eq:vdc}, we only need to compute the second derivative in $k$ of the phase term, i.e.,
\begin{align*}
    \left| \partial _{kk} \left[-t\sqrt{1+k^2}+kr\right] \right| &=  \frac{t}{(1+k^2)^{\frac32}} = t\langle k \rangle ^{-3} \gtrsim t2^{-3(j+2)} \, ,
\end{align*}
since $\psi_j (k)$ is supported on $[2^j, 2^{j+2}]$. Hence, substituting this lower bound as $\lambda$ in \eqref{eq:vdc}, we get the second upper bound in Lemma \ref{lemma:vandercorput}. 

To obtain the third upper bound in Lemma \ref{lemma:vandercorput}, we first extract more time-decay using integration by parts, and then apply Van der Corput's Lemma. Noting that the domain of integration is bounded away from $k=0$, we have
\begin{align*}
    \int\limits_{\R} &e^{-it\sqrt{1+k^2}+ikr} \psi_j (k) \, dk =  \int\limits_{\R} \partial_k [e^{-it\sqrt{1+k^2}}] e^{ikr} \frac{1}{-it}\frac{\sqrt{1+k^2}}{k} \psi_j (k) \, dk \\
    &= \left[e^{-it\sqrt{1+k^2}} e^{ikr} \frac{1}{-it}\frac{\sqrt{1+k^2}}{k} \psi_j (k) \right]_{k=-\infty}^{k=\infty} +\frac{1}{it}\int\limits_{\R} e^{-it\sqrt{1+k^2}} \partial_k \left[e^{ikr} \frac{\sqrt{1+k^2}}{k} \psi_j (k) \right]\, dk\\
    &=\frac{1}{it}\int\limits_{\R} e^{-it\sqrt{1+k^2}+ikr} (ir+\partial_k)\left[\frac{\sqrt{1+k^2}}{k} \psi_j (k) \right]\, dk \, ,
\end{align*}
where the boundary terms vanish since $\psi_j$ is compactly supported. Now we can apply Van der Corput Lemma \eqref{eq:vdc} on the last integral, with the lower bound on the phase term as before, and get the last upper bound in Lemma \ref{lemma:vandercorput}.

\subsection{Proof of Proposition  \ref{lem:bound_highERes}}

We prove $\mathcal{I}_j^1, \mathcal{I}_j^2, \mathcal{I}_j^3\leq C 2^{-(\frac32+\varepsilon) j} \langle x \rangle \langle y \rangle$, for some $C>~0$.
Each term $L(k)$ is based on the explicit expressions for the resolvent kernel Appendix \ref{ap:resolvent_explicit}. It requires the definitions of $A,B,C,D$ as they appear in \eqref{eq:ABCD} (as well as that of $\EE,\FF,\GG,\HH$) and of $\varphi(k,\tau)$ as defined in \eqref{eq:varphik_def}. We find that $L(k)$ is a rational function where both the numerator and the denominator are quadratic in $k$ and $\sqrt{1-k^2}$, and we can write, for some nonzero $a,b,c,d\in~\mathbb{C}$:
$$L(k) = \frac{(ak+b\sqrt{1+k^2})(ck+d\sqrt{1+k^2})}{k\varphi(k,\tau)} \, .$$
Note that, using the asymptotics of $\varphi$, $\EE,\FF,\GG,\HH$ as $k\to \infty$, we have that 
\begin{equation}\label{eq:Lbounds}
    |L(k)|\lesssim {\rm const} \, , \qquad |L'(k)|, |L''(k)| \lesssim k^{-3}\, , \qquad {\rm as}~~ k\to \infty \, . 
\end{equation}

We first bound $\mathcal{I}_j^1$ from above since, as we shall see, it is the largest term (or, the slowest to decay as $k\to \infty$), and will yield the overall upper bound on $I_j$. Note that by the triangle inequality
    \begin{align*}
    \mathcal{I}_j^1 &\lesssim \int_{[k_02^j,k_02^{j+1}]} \langle k\rangle ^{-3/2-\varepsilon}\left( |x\pm y |\cdot|L'(k)| + |L''(k)|\right) \, dk  \numberthis \label{eq:LpLpp_intj} \\
    &+ \int_{[k_02^j,k_02^{j+1}]} \langle k\rangle ^{-3/2-\varepsilon} \left[ \left(|x\pm y|\cdot |L(k)|+ 2|L'(k)| \right)|\chi_j ' (k) |+ |L(k)\chi_j ''(k)| \right] \, dk \, . \numberthis \label{eq:int_chijp}
\end{align*}
Since 
$\langle k \rangle^{-\frac32-\varepsilon}$ is strictly decreasing, and using the upper bounds in \eqref{eq:Lbounds}, we have
$$ |\textrm{\eqref{eq:LpLpp_intj}}| \lesssim 2^{-(\frac32 +\varepsilon) j} \int_{[k_02^j,k_02^{j+1}]} k^{-3} \, dk\  \langle x\rangle \langle y\rangle\, , $$
which vanishes as $j\to \infty$, and so \eqref{eq:LpLpp_intj} is uniformly bounded for all $j \in \mathbb{N}$. To bound \eqref{eq:int_chijp}, we note that the $L'(k)$ vanishes similarly, and so we only have to consider the $L(k)\chi'_j (k)$ and $L(k)\chi _j '' (k)$. 

Since $\chi_j$ is a smooth identifier function, $\chi_j'$ and $\chi_j ''$ are compactly supported, localized around the edges of ${\rm supp}(\chi_ j)$. Hence the value of the integral in \eqref{eq:int_chijp} does not depend on the length of ${\rm supp}\chi_ j$. Since $L(k)$ tends to a constant (see \eqref{eq:Lbounds}), the integral of the terms $|L(k)||\chi ' _k (k)|$ and $|L(k)|^2 |\chi_j ''(x)|$ also tend to a constant as $j\to \infty$. Therefore,  \eqref{eq:int_chijp} is uniformly bounded. Hence, we obtained the desired bound for $\mathcal{I}_j^1$.

The remaining two terms, $\mathcal{I}_j^2$ and $\mathcal{I}_j^3$ (see \eqref{eq:Ij123}), are bounded in a similar manner. In fact, both terms decay even faster, esentially like $k^{-5/2 -\varepsilon}$, and so the overall decay of ${I}_j \lesssim \sum_{\ell = 1,2,3} \mathcal{I}_j ^{\ell}$ is dominated, in high $j$ values (high energy), by that of $\mathcal{I}_j^1$.

%%%%%%%%%%%%%%%%%%%%%%%%%%%%%%%%%%%%%%%%%%%%%%%%%%%%%%%%%%%%%%%%%%%%%%%%%%
\section{Bounds on $F(x;k,\tau)$ and $G(x;k,\tau)$}\label{ap:Fbounds}
%%%%%%%%%%%%%%%%%%%%%%%%%%%%%%%%%%%%%%%%%%%%%%%%%%%%%%%%%%%%%%%%%%%%%%%%%%
Recall the definitions of $F(x;k,\tau)$ in \eqref{eq:Fdef} and $G(x;k,\tau)$ in \eqref{eq:Gdef}
\begin{align*}
    F(x;k,\tau) &\equiv \int_{\R} f(x,y;k,\tau)\langle k\rangle^{-3/2-\eps}\alpha_0(y)dy\ , \\
    f(x,y;k,\tau)&\equiv \begin{pmatrix}\xi_{+,1}(x;k,\tau)\xi_{+,2}(y;-k,\tau)& \xi_{+,1}(x;k,\tau)\xi_{+,1}(y;-k,\tau)\\
    \xi_{+,2}(x;k,\tau)\xi_{+,2}(y;-k,\tau)& \xi_{+,2}(x;k,\tau)\xi_{+,1}(y;-k,\tau)\end{pmatrix}\\
    &+\begin{pmatrix}e^{-i\tau}\xi_{-,1}(x;k,\tau)\xi_{-,2}(y;-k,\tau) & e^{-i\tau}\xi_{-,1}(x;k,\tau)\xi_{-,1}(y;-k,\tau)\\
    e^{-i\tau}\xi_{-,2}(x;k,\tau)\xi_{-,2}(y;-k,\tau ) &     e^{-i\tau}\xi_{-,2}(x;k,\tau)\xi_{-,1}(y;-k,\tau)\end{pmatrix}\ ,\\
    G(x;k,\tau) &\equiv \partial_k\left(\chi(k)\frac{\vert T(k,\tau)\vert^2}{k}F(x;k,\tau)\right) \, .
\end{align*}
\begin{prop}[Bounds on $\partial^j_kF(x;k,\tau)$ and $\partial^j_kG(x;k,\tau)$] \label{lem:Fbounds}
    There exist constants $A_j$, for $j=0,1,2$ and $B_l$ for $l=0,1$, which depend on $k_0$,  such that for every $x\in\R$, $k\in (0,2k_0]$ and $\tau\in(0,\tau)$:
   \begin{align}
       \vert\partial_k^{j} F(x;k,\tau)\vert &\leq A_j\langle x \rangle^j\|\langle y\rangle^j\alpha_0(y) \|_{L^1(\R_y)}  \, \\
        % \vert G(x;k,\tau) \vert &\leq B_0\frac{\langle x\rangle}{k^2+\sin^2(\tau/2)}\|\langle y\rangle\alpha_0(y)\|_{L^1(\R_y)}\ , \\
        \vert \partial_k^l G(x;k,\tau)\vert &\leq B_l\frac{\langle x\rangle^{l+1}}{k^l(k^2+\sin^2(\tau/2))}\|\langle y\rangle^{l+1}\alpha_0(y)\|_{L^1(\R_y)}\, . 
   \end{align}
\end{prop}

\begin{proof}
To bound $F(x;k,\tau)$ and its derivatives, we first bound the Jost solutions (see Proposition~\ref{prop:jostkdef}) $\jostk_{\pm} (x;k,\tau)$  and their derivatives with respect to $k$. Every coordinate of $f$ can be written as a sum of product of Jost solutions. Thus, $F(x;k,\tau)$ can be written as a finite sum of the form
$$F (x;k,\tau) = \sum\limits_{j} \int\limits_{\R} \phi_{j,1}(x;k,\tau) \phi_{j,2}(y;k,\tau) \langle k\rangle^{-3/2-\eps} \alpha(y) \, dy$$
where $\phi_{j,\ell}(z,k,\tau)$ is an expression of the form
\begin{equation}\label{eq:Fphi}\frac{(ak+b\sqrt{1+k^2})(ck+d\sqrt{1+k^2})}{k}e^{\pm ikz}  \, , \qquad \, ,
\end{equation}
for some real parameters $a,b,c,d$. Moreover, uniformly in $z=x$ or $y$,  we have $|\phi(z,k,\tau)|\lesssim \langle k\rangle $ by the explicit form of the resolvent kernel entries (see Appendix \ref{ap:resolvent_explicit}), and it follows that 
$$|F(x;k,\tau)| \leq 
 A_0 \langle k \rangle ^{\frac12-\eps} \int\limits_{\R} |\alpha_0 (y)| \, dy \, . $$
 Since we have constrained $k$ to vary in $[0,2k_0]$, we have the desired upper bound for $F(x;k,\tau)$.
 
 Turning now to $k-$ derivatives of $F(x;k,\tau)$, we note that $k\mapsto F(x;k,\tau)$  is smooth functions of $k$
  and hence for each $j$, $\partial^j_k F(x;k,\tau)$ is uniformly bounded on 
   for $k\in[0,2k_0]$. The bounds on $\partial^j_k F(x;k,\tau)$ involve the weights $\langle x \rangle ^j$ and $\langle y \rangle^j$, due to the appearance of factors of $e^{ikx}$ and $e^{iky}$ in $\phi_{j,l}(x;k,\tau)$; see \eqref{eq:Fphi}. Hence,   $\partial_k^j F(x;k,\tau)$  gives rise, via  $\phi_{j,1} (x;k, \tau)\phi_{j,2}(y;k,\tau)$, to terms whose behavior is bounded by $\langle x\rangle^j\langle y\rangle^j$, leading to the asserted bounds for the derivatives of $F$.

We now turn to the bounds on 
$G(x; k,\tau)$. Expanding the expression, we have
\begin{align}
    G(x;k,\tau) &= \frac{k\chi'(k)-\chi(k)}{k^2}\vert T(k,\tau)\vert^2 F(x;k,\tau) + \frac{\chi(k)}{k}\vert T(k,\tau)\vert^2 \partial_k F(x;k,\tau)\\
    &\qquad + \frac{\chi(k)}{k}\partial_k \vert T(k,\tau)\vert^2 F(x;k,\tau).
\end{align}
Applying the above bounds for $F(x;k,\tau)$  as well as those for $T(k,\tau)$ in \Cref{lem:TransmissionBounds}, we obtain
\begin{align}
   & \vert G(x;k,\tau)\vert \\
    &\quad \lesssim\frac{1}{k^2}\frac{k^2}{k^2+\sin^2(\tau/2)}\|\alpha_0\|_{L^1(\R)}+ \langle x\rangle \|\langle y\rangle \alpha_0(y)\|_{L^1(\R_y)}\  +\frac{1}{k}\frac{2k}{k^2+\sin^2(\tau/2)}\|\alpha_0\|_{L^1(\R)}\\
    &\lesssim \frac{\langle x\rangle}{k^2+\sin^2(\tau/2)}\|\langle y\rangle\alpha_0(y)\|_{L^1(\R_y)},  
\end{align}
with implied constants which are independent of $k$ and $\tau$. The argument for $\partial_k G(x;k,\tau)$ follows similarly using the product rule as well as the bounds on $\partial_k^2 F(x;k,\tau)$ and $\partial_k^2 \vert T(k,\tau)\vert^2$.
This completes the proof of Proposition \ref{lem:Fbounds}.

\end{proof}

\nocite{*}
\printbibliography

\end{document}